\documentclass[reqno, 11pt]{amsart}

\usepackage{tikz}
\usepackage{subcaption}
\usepackage{listofitems}
\usepackage{pgfplots}
\pgfplotsset{compat=1.15}
\usepackage{mathrsfs}
\usetikzlibrary{arrows}
\usepackage{enumerate}
\usepackage{mathtools}
\usepackage[scale]{mnotes}
\usepackage{dsfont}
\usepackage[normalem]{ulem}
\usepackage{tikz}

\usetikzlibrary{patterns}

\usepackage[a4paper,left=35mm,right=37mm,top=37mm,bottom=40mm,marginpar=20mm]{geometry} 

\usepackage[utf8]{inputenc}
\usepackage{amsmath}
\usepackage{amsfonts}
\usepackage{amssymb}
\usepackage{amsthm}
\usepackage{esint}
\usepackage{bbm}
\usepackage{bm}
\usepackage{cancel}
\usepackage{color, comment}
\usepackage[square,sort,comma,numbers]{natbib}
\usepackage[pdftex,     
plainpages=false,   
breaklinks=true,    
colorlinks=true,
linkcolor=blue,
citecolor=blue,
pdftitle={Local Poincare constants and mean oscillation for BV functions},
pdfauthor={Adolfo Arroyo Rabasa, Paolo Bonicatto, Giacomo Del Nin}
]{hyperref} 

\usepackage{graphicx}       
\usepackage{hyperref}       
\makeatletter
\def\@cite#1#2{[\textbf{#1\if@tempswa , #2\fi}]}
\makeatother

\newcommand{\sbullet}{\begin{picture}(1,1)(-0.5,-2)\circle*{2}\end{picture}}
\newcommand{\frarg}{\,\sbullet\,}

\usepackage{xcolor}
\definecolor{Purple1}{HTML}{ED1B23}
\definecolor{Purple2}{HTML}{9600ff}
\definecolor{Purple3}{HTML}{660066}

\newcommand{\bbf}{\mathbf{b}}

\DeclareMathOperator{\Tan}{Tan}

\newcommand{\Gsf}{\mathsf G}
\newcommand{\Ksf}{\mathsf K}

\newcommand{\Fsf}{\mathsf F}

\newcommand{\Acal}{\mathcal{A}}
\newcommand{\Bcal}{\mathcal{B}}

\newcommand{\Fcal}{\mathcal{F}}
\newcommand{\Gcal}{\mathcal{G}}
\newcommand{\Hcal}{\mathcal{H}}

\newcommand{\Lcal}{\mathcal{L}}
\newcommand{\Mcal}{\mathcal{M}}

\newcommand{\Ocal}{\mathcal{O}}
\newcommand{\Pcal}{\mathcal{P}}
\newcommand{\Qcal}{\mathcal{Q}}

\newcommand{\e}{\varepsilon}
\newcommand{\eps}{\varepsilon}
\renewcommand{\H}{\mathcal{H}}

\newcommand{\F}{\mathcal{F}}

\newcommand{\R}{\mathbb{R}}

\newcommand*{\genbf}[1]{\ifmmode\mathbf{#1}\else\textbf{#1}\fi}

\newcommand{\dist}{\mathrm{dist}}
\newcommand{\Osc}{\mathrm{Osc}}
\newcommand{\loc}{\mathrm{loc}}
\newcommand{\weakstarto}{\overset{\ast}{\rightharpoonup}}
\newcommand{\1}{\mathbbm{1}}

\renewcommand{\1}{\mathds1}
\newcommand{\dd}{\mathrm{d}}

\DeclareMathOperator{\spt}{spt}
\DeclareMathOperator{\diverg}{div}

\newtheorem{theorem}{Theorem}[section]
\newtheorem{lemma}[theorem]{Lemma}
\newtheorem{proposition}[theorem]{Proposition}
\newtheorem{corollary}[theorem]{Corollary}

\theoremstyle{definition}
\newtheorem{definition}[theorem]{Definition}
\newtheorem{remark}[theorem]{Remark}
\newtheorem{question}{Question}

\numberwithin{equation}{section}

\DeclareRobustCommand{\SkipTocEntry}[5]{}

\title[Local Poincar\'e constants and mean oscillation for $BV$ functions]{Local Poincar\'e Constants and \\ 
Mean Oscillation Functionals \\ 
for $BV$ functions}

\author{Adolfo Arroyo-Rabasa}
\address[A.\ Arroyo-Rabasa]{Université catholique de Louvain,
Building Marc de Hemptinne,
Chemin du Cyclotron 2,
1348 Louvain-la-Neuve, Belgium}
\email{adolfo.arroyo@uclouvain.be}

\author{Paolo Bonicatto}
\address[P.\ Bonicatto]{Università di Trento,
Dipartimento di Matematica,
Via Sommarive 5, 38123 Trento, Italy}
\email{paolo.bonicatto@unitn.it}

\author{Giacomo Del Nin}
\address[G.\ Del Nin]{Max Planck Institute for Mathematics in the Sciences, Inselstrasse 22, 04103 Leipzig, Germany}
\email{giacomo.delnin@mis.mpg.de}

\begin{document}

\begin{abstract}

	We introduce the concept of local Poincar\'e constant of a $BV$ function as a tool to understand the relation between its mean oscillation and its total variation at small scales. This enables us to study a variant of the BMO-type seminorms on $\eps$-size cubes introduced by Ambrosio, Bourgain, Brezis, and Figalli. More precisely, 
we relax the size constraint by considering a family of functionals that allow cubes of sidelength smaller than or equal to $\eps$. 
 These new functionals converge, as $\eps$ tends to zero, to a local functional defined on $BV$, which can be represented by integration in terms of the local Poincar\'e constant and the total variation. This contrasts with the original functionals, whose limit is defined on $SBV$ and may not exist for functions with a non-trivial Cantor part. 
    Moreover, we characterize the local Poincar\'e constant of a function with a cell-formula given by the maximum mean oscillation of its $BV$ blow-ups. As a corollary of this characterization, we show that the new limit functional extends the original one to all $BV$ functions. 
    Finally, we discuss rigidity properties and other challenging questions relating the local Poincar\'e constant of a function to its fine properties.
    \vspace{4pt}
    

\vspace{4pt}
\vspace{5pt}

\end{abstract}

\maketitle

\makeatletter
\def\@tocline#1#2#3#4#5#6#7{\relax
  \ifnum #1>\c@tocdepth 
  \else
    \par \addpenalty\@secpenalty\addvspace{#2}%
    \begingroup \hyphenpenalty\@M
    \@ifempty{#4}{%
      \@tempdima\csname r@tocindent\number#1\endcsname\relax
    }{%
      \@tempdima#4\relax
    }%
    \parindent\z@ \leftskip#3\relax \advance\leftskip\@tempdima\relax
    \rightskip\@pnumwidth plus4em \parfillskip-\@pnumwidth
    #5\leavevmode\hskip-\@tempdima
      \ifcase #1
       \or\or \hskip 1em \or \hskip 2em \else \hskip 3em \fi%
      #6\nobreak\relax
    \hfill\hbox to\@pnumwidth{\@tocpagenum{#7}}\par
    \nobreak
    \endgroup
  \fi}
\makeatother

\vspace{-1cm}
\setcounter{tocdepth}{2}
{
  \hypersetup{linkcolor=black}
\tableofcontents
}

\section{Introduction}

The Poincaré--Wirtinger inequality for a locally integrable function $f : \Omega \to \R$ on a (sufficiently regular) domain $\Omega \subset \R^n$ reads $$
\Osc(f,\Omega) \le C(\Omega) |Df|(\Omega)\,.
$$ 
This inequality establishes a quantitative relationship between the mean oscillation 
\[
    \Osc(f,\Omega) \coloneqq \fint_\Omega |f(x) - \fint_\Omega f|\,dx\,,
\] 
and the total variation
\[
|Df|(\Omega) \coloneqq \sup \left\{ \int_\Omega f(x)\, \diverg \varphi(x) \, dx : \varphi \in C^1_c(\Omega;\R^n), \|\varphi\|_\infty \le 1 \right\}\,.
\]

The optimal Poincar\'e constant for characteristic functions on the unit open cube $Q_0:=(-\frac12,\frac12)^n$ is $\frac12$, as shown in~\cite[Eqn. (2.2)]{ABBF}. This result can be generalized to all functions of bounded variation using a simple convexity argument based on the coarea formula. In particular, for an open cube $Q \subset \R^n$, a rescaling argument yields the optimal bound
\begin{equation}\label{eq:Poincare_on_cube}
    \ell(Q)^{n-1}\,\Osc(f,Q) \le \frac 12 |Df|(Q)\,, 
\end{equation}
where $\ell(Q)$ is the sidelength of $Q$. We record (cf. 
 Lemma~\ref{lemma:maximizers_poincare}) that equality in this bound is attained only by functions that exhibit a jump-type discontinuity across a hyperplane that bisects the cube and is parallel to one of its sides.

In view of~\eqref{eq:Poincare_on_cube}, one may wonder if it is possible to extract information about the total variation in terms of the mean oscillation (or related difference quotients). With this paradigm in mind, building on pioneering work of Bourgain et al.~\cite{BBM} (see also~\cite{BBM2,brezis2002recognize,davila2002open}), 
there has been a growing interest in characterizing Sobolev and $BV$ functions through BMO-type functionals.  
In this vein, Ambrosio et al.~\cite{ABBF_CRAS,ABBF} proposed to study the limiting behavior of the functionals 
	\begin{equation}\label{eq:def_K}
		\mathsf K_{\e}(f) \coloneqq \e^{n-1} \sup_{\mathcal H_\e}  \sum_{Q \in \mathcal H_\e} \Osc(f,Q)\,, \qquad f \in L^1_\loc(\R^n)\,,
	\end{equation}
    which measure the mean oscillation of a function $f$ over a collection $\Hcal_\eps$ of disjoint $\e$-size cubes with arbitrary orientation (an isotropic variant of the original functionals introduced in~\cite{BBM,BBM2}). 
    They showed that these functionals provide a characterization of sets of finite perimeter and suggested that they \emph{could} be extended to study certain $BV$ functions. This suggestion was particularly influential in sparking several contributions (e.g.,~\cite{AmbrosioComi,ARBDN22,DPFP,farroni2020formula,FMS,FMS2,lahti2023bmo,PS}) concerning the limiting behavior of~\eqref{eq:def_K} as $\eps$ approaches zero. In this regard, Ponce and Spector~\cite{PS}  and Fusco et al.~\cite{FMS2} established a  limiting lower bound for $\Ksf_\eps$, which together with the upper bound given by the Poincar\'e inequality yields  
	\begin{equation}\label{eq:FMS}
		\frac 14 |Df|(\R^n)  \le \liminf_{\eps \to 0}	\mathsf K_{\e}(f)\le \limsup_{\eps \to 0}	\mathsf K_{\e}(f)\le \frac 12 |Df|(\R^n)  \,.
	\end{equation}
Notice that the equi-boundedness of either limit is equivalent to the finiteness of the total variation. Hence, a natural question that arises from this discussion is  to understand those $BV$ functions for which the limit exists. On the one hand, De Philippis et al.~\cite{DPFP} (cf.~\cite{FMS}) proved that the limit exists for all functions $f \in SBV(\R^n)$ of \emph{special bounded variation} (functions with vanishing Cantor part $D^c f$). 
    More precisely, they showed that if $f \in SBV_\loc(\R^n)$, then 
    \begin{equation}\label{eq:DFP}
    \Ksf_0(f) := \lim_{\eps \to 0} \mathsf K_{\e} (f)= \frac14  |D^a f|(\R^n) + \frac 12 |D^j f|(\R^n) \,,
    \end{equation} 
    where $D^af$ and $D^jf$ denote, respectively, the absolutely continuous and the jump part of $Df$.  
   Using a Cantor-type construction (see Appendix \ref{app:Cantor}), 
   we show that there exist $BV$ functions with non-trivial Cantor part for which the limit $\Ksf_0$ does not exist. 
    Heuristically, this occurs because at a Cantor point $x \in \R^n$, the \emph{localized $\eps$-scale Poincar\'e constant}
        \[
         \eps \mapsto   \sup\left\{\eps^{n-1}\frac{\Osc(f,Q)}{|Df|(Q)}:\, x\in Q,\,\ell(Q)=\eps\right\}
        \]
    may oscillate wildly as $\eps$ tends to zero.
 
\addtocontents{toc}{\SkipTocEntry}
    \subsection*{Introducing a new functional: a scale relaxation} 
 To address the incompatibility of the scales that maximize the oscillation at Cantor points, we propose a scale-relaxation of the $\mathsf{K}_\eps$ functionals to a `less than or equal to $\eps$' scale constraint.  
 
Let $\Omega$ be an open subset of $\R^n$. We define a functional on $L^1_\loc(\Omega)$ by setting 
\begin{equation}\label{eq:G_eps_def_intro}
\Gsf_\eps(f,\Omega):=\sup_{\H_{\leq \eps}(\Omega)} \sum_{Q\in \H_{\leq\eps}(\Omega)} \ell(Q)^{n-1} \Osc(f,Q)\,,
\end{equation}
where $\H_{\leq\eps}(\Omega)$ is a family of disjoint open cubes $Q$ contained in $\Omega$ and having sidelength \emph{at most $\eps$}. Observe that the scale factor $\ell(Q)^{n-1}$ dwells now inside the sum, as it depends on each cube; the functional is free to choose the optimal scale at every location. 
This functional can be thought of as a geometric relaxation (from above) of $\Ksf_\eps$, since by construction 
\begin{equation}\label{eq:bound_above}
\Gsf_\eps(f,\R^n) \ge \Ksf_\eps(f), \qquad \text{for all $f \in L^1_\loc(\R^n)$.}
\end{equation}
It is straightforward to verify from the definition that $\eps \mapsto \mathsf{G}_\eps(f,\Omega)$ is non-decreasing.  This allows us to define the pointwise limit
\[
\Gsf(f,\Omega):=\lim_{\eps\to 0} \Gsf_\eps(f,\Omega) \qquad f \in L^1_\loc(\Omega)\,.
\]
A localized version of~\eqref{eq:FMS} and~\eqref{eq:bound_above}, together with Poincar\'e inequality, implies that $\Gsf(f,\Omega)$ is finite if and only if $|Df|(\Omega) < \infty$. Thus, in studying $\Gsf(f,\frarg)$, there is no loss of generality in requiring that $f \in BV_\loc(\R^n)$.  

Our main goal is to establish an integral representation for $\Gsf(f,\Omega)$ in terms of a quantity that we call the \emph{local Poincar\'e constant} of $f$, which we introduce next.

\addtocontents{toc}{\SkipTocEntry}
 \subsection*{Local Poincar\'e constants of \texorpdfstring{$BV$}{BV} functions}   
Let $f \in BV_\loc(\Omega)$. To better understand  the structural properties of $\Gsf(f,\Omega)$, let us express the sum appearing in \eqref{eq:G_eps_def_intro} as
 \begin{equation}\label{eq:quotient_expression}
    \sum_{Q \in \Hcal_{\le \eps}}  P_{f}(Q) |Df|(Q)\,, \qquad P_{f}(Q) \coloneqq \frac{\ell(Q)^{n-1}\Osc(f,Q)}{|Df|(Q)}\,.
 \end{equation} 
Here $P_{f}(Q)$ can be thought of as the Poincar\'e quotient of $f$ on $Q$. By the definition of $\Gsf_\eps$ as a supremum and $\Gsf$ as a limit of $\Gsf_\e$, it is natural to consider the optimal behavior of $P_{f}$ amongst all cubes $Q$, of sidelength at most $\eps$ (as $\e \to 0$) and containing a given point $x$. In fact, as we shall see next, we will need to consider a hierarchy of  infinitesimal Poincar\'e quotients containing additional information about the position of the point: Given $x \in \Omega$ and $\tau \in [0,1]$, we define the \emph{local $\tau$-Poincar\'e  constant of $f$ at $x$} as 
\[
    p_f^\tau(x) \coloneqq \begin{cases}
        \lim_{\eps \to 0} P_f^\tau(x,\eps) & \text{if $x \in \spt(|Df|)$} \\
        0 & \text{if $x \notin \spt(|Df|)$}
    \end{cases} \,,
\]
where $\tau Q$ denotes the concentric $\tau$-contraction of $Q$ (see Figure~\ref{fig:Poincare_constants}) and 
\begin{equation}\label{eq:P^tau} 
P^\tau_f(x,\eps):=\sup\left\{P_f(Q):\, x\in\tau Q,\, \ell(Q)\le \eps\right\}.
\end{equation} 
Observe that $p_f^\tau \ge p_f^\eta$ whenever $\tau 
\ge \eta$.  Moreover, if $f \in {BV}_\loc(\R^n)$, then $p_f^\tau$ defines a non-negative Borel function on $\spt(|Df|)$.

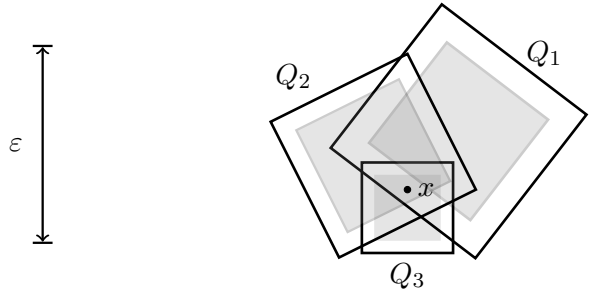
\begin{figure}
    \centering
    	\begin{tikzpicture}[domain=-1.5:1.5,samples=400,scale=1.2]

\begin{scope}[rotate=90]
	\draw[line width=0.35mm, gray, fill=gray, opacity=.2](-.55,-.35) rectangle (.15,.35);
	\draw[line width=0.35mm] (-.70,-.5) rectangle (.30,.5);
    \node at (1.5,-1.5) {$Q_1$}; 
    \node at (1.25,1.25) {$Q_2$}; 
    \node at (-0.95,0) {$Q_3$}; 
	
	\draw[line width=0.35mm] 
	(2.0469676402971926,-0.3765949409126396) -- (0.46069480196423596,0.8414931028539452) -- (-0.7573932418023488,-0.7447797354790111) -- (0.8288795965306076,-1.9628677792455962) -- cycle; 
	
	\draw[line width=0.35mm, fill=gray, opacity=.2] (1.6263135079822613,-0.43182266009759535) -- (0.5159225211491918,0.420838970539014) -- (-0.3367391094874176,-0.6895520162940554) -- (0.7736518773456519,-1.5422136469306649) -- cycle;

    \draw[line width=0.35mm, scale=.75] (0,-1) -- (2,0) -- (1,2) -- (-1,1) -- cycle;
    \draw[line width=0.3mm, fill=gray, opacity=.2, scale=.75] (0.125,-0.625) -- (1.625,0.125) -- (0.875,1.625) -- (-0.625,0.875) -- cycle;
	
		
	\filldraw (0,0) circle (1pt ) node[right] {$x$};

\end{scope}

		
	
  \draw[|<->|,line width=0.3mm] (-4,-0.6) -- (-4,1.6) node [midway,xshift=-10pt]{$\varepsilon$};
  
\end{tikzpicture}
\caption{The shaded area depicts the $\tau$-contraction $\tau Q_i$ of each cube $Q_i$ ($i = 1,2,3$). Only $Q_2$ and $Q_3$ are admissible for $P_f^\tau(x,\eps)$.}
\label{fig:Poincare_constants}
\end{figure}

\addtocontents{toc}{\SkipTocEntry}
\subsection*{The integral representation problem} 
A heuristic reasoning, thinking of \eqref{eq:quotient_expression} as a Riemann sum, suggests the validity of the following integral representation: if $f \in BV_\loc(\Omega)$, then
\begin{equation}\label{eq:conj}
\mathsf{G}(f,\Omega) = \int_\Omega p_f^1(x) \, d|Df|(x)\,.
\end{equation}
Our first main result demonstrates that such an integral representation holds in terms of $\tau$-Poincar\'e constants for $\tau$ sufficiently close to $1$. The precise statement is contained in the following theorem:

\begin{theorem}\label{thm:p_tau_equals_g}
There exists a dimensional constant $\tau(n) \in (0,1)$ with the following property: if $f \in BV_\loc(\Omega)$, then
\[
    \Gsf(f,\Omega) = \int_\Omega p_f^\tau(x) \, d|Df|(x)
\]
for all $\tau(n) \le \tau < 1$. 
\end{theorem}
Our result has the following implications for local Poincar\'e constants: 
Firstly, the local $\tau$-Poincar\'e constants \emph{do not depend} on the particular choice of $\tau \in (0,1)$, provided that $\tau$ is sufficiently close to $1$. Secondly, since $\Ksf(f) \le \Gsf(f,\R^n)$, then  
\begin{equation}\label{eq:range_p}
     \frac 14 \le p_f^\tau 
     \le p_f^1 \le \frac 12 \quad \text{$|Df|$-almost everywhere.}
\end{equation} 
In Theorem~\ref{thm:p_tau_equals_g}, the restriction $\tau \in [\tau(n),1)$ stems mostly from the techniques we adopt, e.g., the need of suitable covering theorems with \emph{uncentered} cubes (see Remark \ref{rmk:limitation_of_tau_n} for further details). Unfortunately, at the moment we are unable to tell if this restriction is purely technical. In view of this discussion one may wonder if for all $\tau \ge \tau(n)$,
\begin{equation}\label{eq:conj2}
     p_f^\tau =  p_f^1  \quad \text{$|Df|$-almost everwhere in $\Omega$.}
\end{equation}
In dimension one we are able to establish this identity through the validity of~\eqref{eq:conj}:
\begin{theorem}\label{thm:one_dim}
    Let $f\in {BV}_\loc(I)$, where $I \subset \R$ is open. Then
    \[
    \Gsf(f,I)=\int_I p_f^1(x)\, d|Df|(x)\,.
    \]
    In particular, for all $\tau(1) \le \tau < 1$,
    \[
        p_f^\tau = p_f^1   \quad \text{$|Df|$-almost everywhere.}
    \]
\end{theorem}
\noindent The validity of \eqref{eq:conj} and \eqref{eq:conj2} for dimensions $n>1$ remains an open question. 

\addtocontents{toc}{\SkipTocEntry}
\subsection*{Tangent oscillations: a representation formula} 

The  original motivation behind our work was to  find a way to retrieve information about the pointwise limit $\Ksf_0(f)$  from the behavior of the mean oscillation of $f$ at infinitesimally small scales. Our idea was to study the mean oscillation of the $BV$ blow-ups (see~\cite[pag. 187]{AFP}) of $f$, at points in the support of $|Df|$.

Delving into this and motivated by the definition of $\Gsf$, we introduce a suitable concept of $BV$ rescaling that takes into account the geometry of cubes: given an open cube $Q$ with $|Df|(Q) > 0$ and an affine bijection $T_Q : Q_0 \to Q$ we consider the normalized function (see Section~\ref{sec:tangents} for the precise definition, which uses \emph{oriented} cubes) 
\[
f_Q(y) \coloneqq  \frac{1}{|Df|(Q)} \, \left( f(T_Q y) - \fint_Q f \right)\, \ell(Q)^{n-1},\qquad y \in Q_0\,.
\]
Now, in order to analyze the infinitesimal behavior of $p_f^\tau$, we also introduce a notion of uncentered $BV$ blow-up tailored for cubes: Given $x \in \spt(|Df|)$, we define the set $\Tan^\tau(f,x)$, of \emph{uncentered $\tau$-tangents of $f$ at $x$}, as the set of all limit points in $L^1(Q_0)$ of the family
$\{f_Q : x \in \tau Q\big\}$ as $\ell(Q) \to 0$.

With these basic concepts at hand, we are now in the position to state our second main result, which gives a description of the Poincar\'e constant in terms of the largest mean oscillations of $\tau$-tangents for $0 \le \tau < 1$:  
\begin{theorem}\label{thm:p_tau_equal_osc}
    Let $f\in BV_\loc(\Omega)$ and let $\tau\in[0,1)$. Then,
    \[
p_f^\tau(x) = \sup_{u\in\Tan^\tau(f,x)} \Osc(u,Q_0)
    \]
    for every $x\in\spt(|Df|)$.
\end{theorem}

For $f \in SBV_\loc$, the largest mean oscillation over $\tau$-tangents is 1/4 at $|D^af|$-a.e. point and 1/2 at $|D^jf|$-a.e. point, regardless of the choice of $\tau$ (see Section \ref{sec:SBV_tangents}). Combining this observation with Theorem~\ref{thm:p_tau_equals_g} and Theorem~\ref{thm:p_tau_equal_osc} we discover that the functional $\Gsf(\frarg,\R^n)$ coincides with $\Ksf_0$ on $SBV_\loc$ functions and can, therefore, be considered as a relaxation from above of $\Ksf_0$: 
\begin{corollary}\label{cor:G_equals_K0_on_SBV}
     Let $f \in SBV_\loc(\R^n)$. Then 
     \[
         \Gsf(f,\R^n) = \Ksf_0(f) = \frac 14 |D^a f|(\R^n) + \frac12 |D^j f|(\R^n)\,.
     \]
 \end{corollary}

Lastly, we record  an interesting consequence of Theorem~\ref{thm:p_tau_equal_osc}. In dimension $n = 1$, the local Poincar\'e constant $p_f^1$ can be expressed as the largest BMO-seminorm
    \[
    \|u\|_{\mathrm{BMO}(Q_0)} \coloneqq \sup \left\{ \Osc(u,Q) : \text{$Q \subset Q_0$ open cube}\right\}\,,
    \]
 amongst all $1$-tangents:
\begin{corollary}\label{cor:BMO}
    Let $I \subset \R$ be an open interval and let $f \in BV_\loc(I)$. Then,
    \[
        p_f^1(x) = \sup_{u\in\Tan^1(f,x)}\|u\|_{\mathrm{BMO}((-\frac 12,\frac 12))}\,.
    \]
for $|Df|$-almost every $x \in I$.
\end{corollary}

\addtocontents{toc}{\SkipTocEntry}
\subsection*{Comments on other shapes} The fact that we work with cubes does not play a determining role in the essence of our statements. For instance, one could work with shapes ---such as a fixed convex shape--- that permit the use of covering theorems (see also \cite{AmbrosioComi,farroni2020formula} for related results). It is worth remarking that, in stark contrast with cubes, the (local) Poincar\'e constants associated with other shapes might depend on the dimension $n$.

\addtocontents{toc}{\SkipTocEntry}
\subsection*{Structure of the paper} Let us briefly summarize the contents of this work. In Section \ref{sec:integral_representation} we prove that $\Gsf(f,\cdot)$ extends to a Borel measure. In Section \ref{sec:lower_bound} we prove the inequality “$\ge$” in Theorem \ref{thm:p_tau_equals_g}. In Section \ref{sec:good_family} we do some preparatory work for the other inequality, proving the existence of good families of cubes. In Section \ref{sec:upper_bound} we prove the inequality “$\le$” in Theorem \ref{thm:p_tau_equals_g}. In Section \ref{sec:tangents} we prove Theorem \ref{thm:p_tau_equal_osc}. In Section \ref{sec:rigidity} we study the properties of sets where $p_f^\tau$ coincides with one of the extremal values $\tfrac14$ and $\tfrac12$. In Appendix \ref{app:measure_theory} we collect some measure-theoretic results and covering theorems. Finally, in Appendix \ref{app:Cantor} we give an example showing that $\Ksf_0$ might fail to exist on $BV$ functions with a non-trivial Cantor part.

\addtocontents{toc}{\SkipTocEntry}
\subsection*{Acknowledgements} This research was supported through the program “Oberwolfach Research Fellows”, by the Mathematisches Forschungsinstitut Oberwolfach, during a research visit by the authors in November 2022. The authors are sincerely grateful for the institute's work-motivating atmosphere and warm hospitality. AAR was supported by the Fonds de la Recherche Scientifique - FNRS under Grant No 40005112. PB and GDN received funding from the European Research Council (ERC) under the European Union's Horizon 2020 research and innovation programme, grant agreement No 757254 (SINGULARITY).

\addtocontents{toc}{\SkipTocEntry}
\subsection*{Notation}
Throughout the paper we adopt the following conventions.
We let $n \ge 1$ be an integer and we always consider $\Omega$ an open subset of $\R^n$. 
We will generically denote (open) cubes in $\R^n$ by the letter $Q$, while $Q_0$ is the centered unit cube $(-\frac 12,\frac 12)^n$. The sidelength of a cube $Q$ will be denoted by $\ell(Q)$ and its center point by $x_Q$. The letter $\tau$ will denote a positive constant between $0$ and $1$, often times attaining the value $1$. With a possible abuse of notation, we shall write $
\tau Q$ to denote the cube of sidelength $\tau\ell(Q)$, centered at $x_Q$, and whose faces are parallel to the ones of $Q$. As usual, we write $B(x,r)$ to denote an open ball of radius $r$ centered at a point $x$.
If $A$ is a set, we denote by $A^c$ its complement and by $\1_A$ its characteristic function. The $n$-dimensional Lebesgue measure
of a set $A \subset \R^n$ is denoted by $\Lcal^n(A)$ or, when no risk of confusion arises, simply by $|A|$. 

\vspace{0.5cm}

\section{Integral Representation}\label{sec:integral_representation}

We start by defining a localized version of the main functionals. Let $\Ocal(\Omega)$ be the family of all open sets contained in $\Omega$, and let $\Gsf_\eps(f,\,\frarg\,) : \Ocal(\Omega) \to [0,\infty]$ be defined as 
\[
\Gsf_\eps(f,U)\coloneqq \sup_{\H_{\le\eps}(U)} \sum_{Q\in \H_{\leq\eps}(U)}\ell(Q)^{n-1} \fint_{Q}\Big|f(x)-\fint_{Q} f\Big|\, \dd x\,, \quad U \in \Ocal(\Omega),
\]
where the supremum is taken among all families of essentially disjoint cubes of sidelength at most $\eps$ contained in $U$. We also define
\[
\Gsf(f,U)\coloneqq \lim_{\eps\to 0} \Gsf_\eps (f,U),\qquad U\in\Ocal(\Omega),
\]
where the limit always exists by monotonicity. Next, we show that $\Gsf(f,\frarg)$ is the restriction of a locally bounded Borel measure

\begin{proposition}\label{prop:G_is_measure}
If $f\in BV_\loc(\Omega)$, the set function $\Gsf(f,\,\frarg\,):\Ocal(\Omega)\to[0,\infty]$ can be extended to a Radon measure that is absolutely continuous with respect to $|Df|$. In particular, there exists a nonnegative function $g_f : \Omega \to [0,\infty)$ such that
\[
   \Gsf(f,B) = \int_B g_f(x) \, \dd |Df|(x) 
\]
for every Borel set $B \subset \Omega$. 
\end{proposition}

\begin{proof} Let us fix $f\in BV_\loc(\Omega)$ and set $\Fsf(\cdot) := \Gsf(f,\cdot)$.
The Poincar\'e inequality yields the (pointwise) inequality 
\begin{equation}\label{eq:bound_with_poincare}
\Fsf(U) \le \frac12|Df|(U), \quad \forall U \in \Ocal(\Omega),
\end{equation}
which implies that $\Fsf$ vanishes on $|Df|$-null (open) sets. 
Notice that the assertions that $\Fsf$ is absolutely continuous with respect to $|Du|$ and the existence of the density $g_f$ will follow from \eqref{eq:bound_with_poincare} once we know that $\Fsf$ is (the restriction of) a measure (to open sets). In order to verify this, we show that Theorem \ref{thm:De_Giorgi_Letta} can be applied to the set function $\Fsf$.

Clearly, $\Fsf$ is increasing and the condition $\Fsf(\emptyset)=0$ is satisfied.  
Let us prove that the three remaining assumptions of Theorem \ref{thm:De_Giorgi_Letta} are satisfied:

\begin{enumerate}
    \item[(i)] \emph{Subadditivity.}
     Fix $\eta,\eps >0$ and let
     $\Qcal := \Hcal_{\le \e}( U\cup V)$ be a family of disjoint cubes contained in $U\cup V$ with sidelength $\ell(Q) \le \eps$. Let $U_\eta:= \{x \in U: \dist(x,\partial U)>\eta\}$ and let $V_\eta$ be defined analogously.  
     We split the family $\Qcal$ into three subfamilies:
     \begin{align*}
         \Qcal_{U_\eta}&:= \{Q \in \Qcal: Q \subset U_\eta\},\\
         \Qcal_{V_\eta}&:= \{Q \in \Qcal\setminus \Qcal_{U_\eta}: Q \subset V_\eta\},\\
         \Qcal_{R} &:= \Qcal\setminus (\Qcal_{U_\eta} \cup \Qcal_{V_\eta}).
     \end{align*}
     Observe that $Q\in \Qcal_{R}$ implies 
     \[
    Q \subset (U\setminus U_{\eta+\eps\sqrt{n}}) \cup (V\setminus V_{\eta+\eps\sqrt{n}}) .
     \]
     Therefore we have 
    \begin{align*}
    \sum_{Q\in \Qcal} \ell(Q)^{n-1} \Osc(f,Q) & \le  \Gsf_\eps(f,U_\eta)+\Gsf_\eps(f,V_\eta)+ \sum_{Q\in \Qcal_R} \ell(Q)^{n-1} \Osc(f,Q) \\ 
     \le \Gsf_\eps(f,U) &+\Gsf_\eps(f,V)+\frac12 |Df|((U\setminus U_{\eta+\eps\sqrt{n}}) \cup (V\setminus V_{\eta+\eps\sqrt{n}})).
    \end{align*}
    Taking the supremum in the LHS and sending $\eta \to 0$ and $\eps \to 0$ we get 
    \[
    \Fsf(U\cup V) \le \Fsf(U)+\Fsf(V), 
    \]
    as desired.
    \item[(ii)] \emph{Superadditivity.} We now show superadditivity holds for $\Gsf_\eps$ on disjoint sets. Let $\eps>0$ and $U,V \in \Ocal(\Omega)$ be open, disjoint sets. Pick two arbitrary families of cubes $\Hcal_{\le \e(U)}$ and $\Hcal_{\le \e(V)}$ as in the definition of $\Gsf_\eps$. 
    Then the union of these families is an admissible competitor for $\Gsf_\e(f,U\cup V)$. This implies
    \[
    \Gsf_\e(f,U\cup V) \ge \sum_{Q\in \H_{\leq\eps(U)}} \ell(Q)^{n-1} \Osc(f,Q) + \sum_{Q'\in \H_{\leq\eps(V)}} \ell(Q)^{n-1} \Osc(f,Q).
    \]
    Passing to the supremum in the RHS first over all families $\Hcal_{\leq\eps(U)}$ and then over $ \Hcal_{\leq\eps(V)}$ we deduce the desired supperadditivity
    \[
    \Gsf_\e(f,U\cup V) \ge 
    \Gsf_\e(f,U) + \Gsf_\e(f,V)\,.
    \]
    Letting $\eps \to 0$, we recover the superadditivity for $\Fsf$.  
    \item[(iii)] \emph{Inner regularity.} This property follows from subadditivity in the following way. Fix $U$ open set, and $\eta>0$. We suppose that $U$ is bounded (otherwise a similar argument works with minor changes) and write $U=U_\eta\cup (U\setminus \overline{U_{2\eta}})$. Then, $U_\eta\subset\subset U$ and
    \[
    \Fsf(U)\le \Fsf(U_\eta)+\Fsf(U\setminus \overline{U_{2\eta}})\le \Fsf(U_\eta)+\frac12 |Df|(U\setminus \overline{U_{2\eta}}).
    \]
    Since $|Df|(U\setminus \overline{U_{2\eta}})\to 0$ as $\eta\to 0$, we obtain the desired inner regularity property.\qedhere
\end{enumerate} 
\end{proof}

\begin{remark}\label{rmk:bounds_g_f}
    By \eqref{eq:FMS}, \eqref{eq:bound_above}, and Poincar\'e inequality the density $g_f$ of the functional $\Gsf$ satisfies
    \begin{equation}\label{eq:bounds_g_f}
        \frac14\le g_f(x)\le \frac12\qquad\text{at $|Df|$-almost every $x$.}
    \end{equation}
\end{remark}

\vspace{0.5cm}

\section{Lower bound}\label{sec:lower_bound}

In all that follows we shall work with $f \in BV_\loc(\Omega)$, where $\Omega \subset \R^n$ is an open set.

\begin{lemma}\label{lemma:g_ge_p_tau} Let $\tau\in [0,1)$. For $|Df|$-almost every $x \in \Omega$ it holds
\[
    g_f(x) \ge p^\tau_f(x).
\]
If $n=1$, then the same conclusion holds also for $\tau=1$ (see Remark \ref{rmk:one_dim_tau_1}).
\end{lemma}

\begin{proof}
Let $\eps, \delta > 0$. For each point $x \in \spt(|Df|)$, let $\Fcal_\eps(x)$ be the family of cubes $Q\subset \Omega$ satisfying the following properties:
\begin{enumerate}[$(a)$]
    \item $x \in \tau Q$;
    \item $\ell(Q) \le \eps$;
    \item The following approximate continuity estimate holds:
    \[
        \fint_{Q} |p_f^\tau(x) - p_f^\tau(y)| \, d|Df|(y) \le \delta\,.
    \]
    \item The following almost maximizing property holds:
    \[
        \frac{\Osc(f,Q)\ell(Q)^{n-1}}{|Df|(Q)} > p_f^\tau(x) - \delta\,.
    \]
    \item $|Df|(\partial Q)=0$.
\end{enumerate}
Notice that if a cube $Q$ satisfies (a)-(b)-(c)-(d) and $|Df|(\partial Q)>0$, we can slightly shrink $Q$ so that it also satisfies (e). Then by the definition of $p_f^\tau(x)$ and by Lebesgue's differentiation theorem (see Theorem \ref{thm:Lebesgue_diff_tau_cubes_functions}), the family 
    \[
    \Fcal_\eps \coloneqq \bigcup_{x\in \spt(|Df|)}\{\overline{Q}:\, Q\in \Fcal_\eps(x)\}
    \]
    defines a fine cover of $\spt(|Df|)$. By property (a) we can invoke Theorem \ref{thm:tau_vitali}, applied with the measure $\mu^*=|Df|$, to obtain a countable subcollection $ \mathcal F'_{\eps}\subset \mathcal F_{\eps}$ of cubes with disjoint closures satisfying
\begin{equation}\label{eq:Vitali_property_lower_bound}
|Df|\bigg(\Omega \setminus \bigcup_{Q \in \mathcal F'_{\eps}} \overline{Q}\bigg) = 0. 
\end{equation}
We now use $\mathcal{F}_{\eps}'$ as a competitor collection in the definition of $\Gsf_\eps$ to obtain a lower bound estimate:
\begin{align*}
\Gsf_\eps(f,\Omega) & \ge \sum_{Q \in \mathcal F'_{\eps}} \Osc(f,Q)\ell(Q)^{n-1} \\ 
& \overset{(d)}{\ge} \sum_{Q \in \mathcal F'_{\eps}} (p^\tau_f(\bar x_Q)-\delta)|Df|(Q)\\
& \overset{(c)}{\ge} \sum_{Q \in \mathcal F'_{\eps}}\Big(\fint_Q p_f^\tau(y)\,d|Df|(y)-2\delta\Big) |Df|(Q)\\
& = \int_\Omega p_f^\tau(y)\, d|Df|(y)-2\delta |Df|(\Omega)\,.
\end{align*}
In passing to the last equality we used that $|Df|(\partial Q)=0$ for every $Q\in \Fcal_\eps$, together with the Vitali property \eqref{eq:Vitali_property_lower_bound}. Considering first the limit as $\delta\to 0$ and then the limit as $\eps\to 0$ we deduce
\[
\Gsf(f,\Omega)\ge \int_\Omega p_f^\tau(y)\, d|Df|(y).
\]

Since this holds for every $\Omega$, an application of Radon-Nikodym implies that $g_f(x)\ge p_f^\tau(x)$ for $|Df|$-almost every $x$. 
This finishes the proof.
\end{proof}

\begin{remark}[On the restriction $\tau < 1$]\label{rmk:one_dim_tau_1}
In our proof, the restriction $\tau<1$ is essential for the lower bound, as it relies on the availability of a covering theorem. In the one-dimensional case, the proof above applies also to the case $\tau=1$, as Theorem \ref{thm:vitali_general} generalizes Theorem \ref{thm:tau_vitali} to this setting.
\end{remark}

\vspace{0.5cm}

\section{Good families of cubes}\label{sec:good_family}

This section aims to establish the existence of families of cubes that exhibit near-optimality for $\Gsf_\eps$ and possess additional desirable properties. These families will play a pivotal role in the proof of the upper bound, presented in the subsequent section.  The precise statement is the following.

\begin{proposition}[Good families of cubes]\label{prop:good_family}
There exists a dimensional constant $\tau(n)\in (0,1)$ with the following property: If $\tau\in[\tau(n),1)$, then, for each  $\eps,\delta>0$, there exists a family $\Fcal \in \Hcal_{\le \eps}(\Omega)$ satisfying the following properties:
\begin{enumerate}[(i)]
    \item $\Fcal$ is $\delta$-almost maximizing for $\Gsf_\eps$ in the sense that
    \[
        \sum_{Q\in\F} \ell(Q)^{n-1}\, \Osc(f,Q)\ge(1-\delta)\,\Gsf_\eps(f,\Omega)\,.
    \]
    \item Every $Q \in \Fcal$ satisfies
    \[
        \Osc(f,Q)\ell(Q)^{n-1}\geq \frac18 |Df|(Q).
    \]
    \item Every $Q\in\Fcal$ satisfies
    \[
    |Df|(\tau Q)\ge \frac{1}{16} |Df|(Q).
    \]
\end{enumerate}
\end{proposition}

In order to prove the proposition we need a few preliminary results.
First, we will need the following version of the Poincar\'e--Wirtinger inequality.

\begin{lemma}[Modified Poincar\'e–Wirtinger]\label{lemma:modified_poincare} Let $\tau \in (0,1]$ and let $Q$ be any given cube in $\R^n$. There exists a constant  $C(n,\tau)$ such that
    \[
    \left\|f-\fint_{\tau Q}f\right\|_{L^{1^*}(Q)}\leq C (n,\tau) |Df|(Q)
    \]
    for all $f \in BV(Q)$. 
\end{lemma}
\begin{proof} By scaling we can assume that $Q$ has volume one. By a standard approximation result for $BV$ maps, it suffices to prove the statement for $f \in C^1(\overline Q)$. 
For each $t \in (0,1)$, we will use the change of variables $y_t = tx + (1-t)\tau x$, which maps $Q$ into $[\tau + t(1 - \tau)]Q$. With this in mind, we apply the fundamental theorem of calculus to deduce the estimate
    \begin{align*}
        \left|\fint_Q f - \fint_{\tau Q}f\right| & = \left|\int_Q f(x) - f(\tau x) \, dx\right| \\
        & = \left| \int_Q \int_0^1 \nabla f(tx + (1-t)\tau x) \cdot (1 - \tau)x \,dt\,dx\right|\\
        & \le (1 - \tau)  \int_Q \int_0^1 |x| |\nabla f|(tx + (1-t)\tau x) \, dt \, dx \\
        & \le (1 - \tau) \int_0^1  \int_{[t+(1-t)\tau]Q} \frac{|y_t|}{|\tau + t(1 - \tau)|^{n+1}} |\nabla f|(y_t)\, dy_t \, dt \\
        & \le \frac{\sqrt{n}}{2}(1 - \tau) \int_0^1 \frac{| Df|([t+(1-t)\tau]Q)}{|\tau + t(1 - \tau)|^{n+1}} \, dt\\
        & \le\frac{\sqrt{n}}{2} \frac{1-\tau}{\tau^{n+1}}|Df|(Q)\,.
    \end{align*}
    In particular, this allows us to estimate the $L^1$-difference between the averages on the large and small cubes as
    \[
        \left\| \fint_Q f - \fint_{\tau Q} f\right\|_{L^{1^*}(Q)} \le \frac{\sqrt{n}}{2} \frac{1-\tau}{\tau^{n+1}}|Df|(Q)\,. 
    \]
    In light of this estimate and the triangle inequality, we further get
    \begin{align*}
        \left\| f - \fint_{\tau Q}f\right\|_{L^{1^*}(Q)} & \le \left\| f - \fint_Q f\right\|_{L^{1^*}(Q)} + \left\| \fint_Q f - \fint_{\tau Q} f \right\|_{L^{1^*}(Q)} \\
        & \le \left(C(n) +  \frac{\sqrt{n}}{2} \frac{1-\tau}{\tau^{n+1}} \right)|Df|(Q), 
    \end{align*}
    where $C(n)$ is the classical Poincaré-Wirtinger constant for functions in $L^{1^*}(Q)$. 
    This finishes the proof.
\end{proof}

Next, we prove two technical lemmas that quantify the amount of mean oscillation on functions whose derivative concentrates near the boundary of a cube.

\begin{lemma}\label{lemma:almost_constant}
    Let $\tau \in (0,1)$, $\eta>0$ and let $Q$ be a given cube in $\R^n$ such that 
    \[
    |Df|(\tau Q)\leq \eta |Df|(Q)\,.
    \]
    Then,  
    \[
    \frac{\Osc(f,Q)}{|Df|(Q)}\ell(Q)^{n-1}\leq \frac{\eta}{2}+3C(n,\tau)(1-\tau)\,, 
    \]
    where $C(n,\tau)$ is the constant of Lemma~\ref{lemma:modified_poincare}.
\end{lemma}

\begin{proof}
    Since the left-hand side is invariant under translations, dilations, and multiplications by constants, there is no loss of generality in assuming that $Q$ is the unit cube, $|Df|(Q)=1$, and $\int_{\tau Q}f=0$. By Lemma \ref{lemma:modified_poincare} we get 
    \[ \left(\int_{Q}|f|^{1^*}\right)^{1/1^*}=\left(\int_{Q}\left\vert f-\fint_{\tau Q}f\right\vert ^{1^*}\right)^{1/1^*}\leq C(n,\tau) |Df|(Q)=C(n,\tau).
    \]
    From this estimate we deduce the following auxiliary estimates: First, by H\"older's inequality we get
    \begin{align*}
    \left|\int_{Q\setminus \tau Q}f\right|&\le\int_{Q\setminus \tau Q}|f|\\
    & \le \left(\int_{Q\setminus \tau Q} |f|^{1^*}\right)^{1/1^*} |Q\setminus \tau Q|^{1-1/1^*}\le C(n,\tau) (1-\tau).
    \end{align*}
    On the other hand, the triangle inequality gives
    \[
    \left|\int_Q f\right|\le \left|\int_{\tau Q} f\right|+\left|\int_{Q\setminus \tau Q}f\right|\le C(n,\tau)(1-\tau)\,. 
    \]
    Breaking $\Osc(f,Q)$ into smaller quantities that can be bounded by the quantities above, we obtain the sought estimate:
    \begin{align*}
        \Osc(f,Q)&=\int_Q\left|f-\int_Q f\right|=\int_{\tau Q}\left|f-\int_Q f\right|+\int_{Q\setminus \tau Q}\left|f-\int_Q f\right|\\
        & \le \int_{\tau Q}\left|f-\fint_{\tau Q}f\right| + \int_{\tau Q} \left|\fint_{\tau Q}f-\int_Q f\right|+\int_{Q\setminus \tau Q}|f| + \int_{Q\setminus \tau Q}\left|\int_Q f\right|\\
        &\le |\tau Q| \,\frac{1}{2}|Df|(\tau Q)+|\tau Q| C(n,\tau)(1-\tau) \\
        & \qquad +C(n,\tau)(1-\tau)+C(n,\tau)(1-\tau)^{n+1}\\
        &\le \frac{\eta}{2}+3C(n,\tau)(1-\tau).
    \end{align*}
    This finishes the proof.
\end{proof}

\begin{corollary}\label{corollary:burrata}
    There exists a dimensional constant $\tau(n) \in  (0,1)$ with the following property: if $\tau\in[\tau(n),1)$ and if $f \in BV(Q)$ is such that 
    \[
    |Df|(\tau Q)\leq \frac{1}{16} |Df|(Q)\,,
    \]
    then
    \[
    \Osc(f,Q)\ell(Q)^{n-1}<\frac18|Df|(Q)\,.
    \]
\end{corollary}

\begin{proof}
It suffices to apply Lemma \ref{lemma:almost_constant} with $\eta=\frac{1}{16}$ and notice that for $\tau$ sufficiently close to $1$ it holds $3C(n,\tau)(1-\tau)<\tfrac{1}{16}$. 
\end{proof}

We are now ready to prove Proposition \ref{prop:good_family}.

\begin{proof}[Proof of Proposition \ref{prop:good_family}]
    Let $\eta\in (0,\tfrac12)$. By the definition of $\Gsf_\eps$, we can find a disjoint collection $\Qcal \in \Hcal_{\le \eps}(\Omega)$ satisfying
    \begin{equation}\label{eq:eta1}
    \sum_{Q\in \Qcal}\Osc(f,Q)\ell(Q)^{n-1}\ge (1-\eta)\Gsf_\eps(f,\Omega)\,.
    \end{equation}
    We then define a ``good'' subfamily of $\Qcal$ by setting 
    \[
    \Gcal:=\left\{Q\in\Qcal : \Osc(f,Q)\ell(Q)^{n-1}\ge \tfrac18|Df|(Q)\right\}\,.
    \]
    Accordingly, we call $\Bcal:=\Qcal\setminus \Gcal$ the ``bad'' subfamily of $\Qcal$. For each $Q\in \Bcal$, we can select a collection $\Pcal(Q)$ of pairwise disjoint cubes, contained in $Q$ and satisfying 
    \begin{equation}\label{eq:eta3}
        \sum_{Q'\in\Pcal(Q)}\Osc(f,Q')\ell(Q')^{n-1}\ge (1-\eta)\Gsf_\eps(f,Q)\,.
    \end{equation}
    Note that the family $\widetilde\Qcal:=\Gcal\cup \bigcup_{Q\in\Bcal}\Pcal(Q)$ is still an admissible family of disjoint cubes and hence, by definition,
    \begin{equation}\label{eq:eta2}
        \sum_{Q\in\widetilde\Qcal}\Osc(f,Q)\ell(Q)^{n-1}\le \Gsf_\eps(f,\Omega)\,.
    \end{equation}
    From~\eqref{eq:eta1}-\eqref{eq:eta2} and the definition of $\Qcal$ we infer that
    \begin{align*}
        \eta \Gsf_\eps(f,\Omega)&\ge \sum_{Q\in\widetilde\Qcal}\Osc(f,Q)\ell(Q)^{n-1}-\sum_{Q\in \Qcal}\Osc(f,Q)\ell(Q)^{n-1}\\
        &=\sum_{Q\in\Bcal}\left(\sum_{Q'\in \Pcal(Q)}\Osc(f,Q)\ell(Q)^{n-1}-\Osc(f,Q)\ell(Q)^{n-1}\right)\\
        &\stackrel{\eqref{eq:eta3}}\ge \sum_{Q\in\Bcal} ((1-\eta)\Gsf_\eps(f,Q)-\tfrac18 |Df|(Q)) \eqqcolon A \,.
    \end{align*}
    Using the lower bound $\tfrac14|Df|(Q)\leq \Gsf(f,Q)\le\Gsf_\eps(f,Q)$ (see \eqref{eq:FMS}) and the Poincar\'e inequality $\Osc(f,Q)\ell(Q)^{n-1}\le \tfrac12 |Df|(Q)$ we further estimate 
    \begin{equation}\label{eq:eta4}
      \begin{split}
          \eta \Gsf_\eps(f,\Omega) \ge A& \ge  \sum_{Q\in\Bcal} (\tfrac14(1-\eta)-\tfrac18) |Df|(Q) \\
        & = \tfrac18(1-2\eta)\sum_{Q\in\Bcal}  |Df|(Q)\\
        &\ge \tfrac14(1-2\eta)\sum_{Q\in\Bcal}  \Osc(f,Q)\ell(Q)^{n-1}\,.
      \end{split}   
\end{equation}
    By construction, the collection $\Gcal$ satisfies the following properties:
    \begin{enumerate}[(a)]
        \item $\ell(Q)\le\eps$ for every $Q\in \Gcal$;
        \item it satisfies
        \begin{align*}
            \sum_{Q\in\Gcal} \Osc(f,Q)\ell(Q)^{n-1}&= \sum_{Q\in\Qcal} \Osc(f,Q)\ell(Q)^{n-1}-\sum_{Q\in\Bcal} \Osc(f,Q)\ell(Q)^{n-1}\\
            &\stackrel{\eqref{eq:eta1}-\eqref{eq:eta4}}\ge (1-\eta)\Gsf_\eps(f,\Omega)-\frac{4\eta}{1-2\eta}\Gsf_\eps(f,\Omega) \\
            &=\left(1-\eta\frac{5-2\eta}{1-2\eta}\right) \Gsf_\eps(f,\Omega)\,;
        \end{align*}
        \item and $\Osc(f,Q)\ell(Q)^{n-1}\ge \tfrac18|Df|(Q)$ for every $Q\in\Gcal$ (by the definition of good subfamily).
    \end{enumerate}
    Choosing $\eta$ small enough so that $\eta\frac{5-2\eta}{1-2\eta}\le\delta$ we find out that $\Gcal$ satisfies~$(i)$ and~$(ii)$. Finally, $(iii)$ follows from $(ii)$ by applying Corollary \ref{corollary:burrata}.
\end{proof} 

\vspace{0.5cm}

\section{Upper bound} \label{sec:upper_bound}
In this section we prove that $p^\tau_f(x)$ bounds $g_f(x)$ from above for $|Df|$-almost every $x \in \R$.

\begin{lemma}\label{lemma:g_le_p_tau} Let $\tau(n) \le \tau \le 1$, where $\tau(n)$ is the constant defined in Corollary \ref{corollary:burrata}. Then, for $|Df|$-almost every $x \in \Omega$ it holds
\begin{equation}\label{eq:upper_bound}
    g_f(x) \le p^\tau_f(x).
\end{equation}
\end{lemma}

\begin{proof} 
We begin by recalling the definition 
\[
P_f^\tau(x,\rho):= \sup_{\substack{y\in \tau Q\\ \ell(Q) \le \rho}}\frac{\Osc(f,Q)}{|Df|(Q)}\ell(Q)^{n-1}. 
\]
Given $\delta,\rho>0$ define the set 
\begin{equation}\label{eq:E_rho_delta_definition}
E_{\rho,\delta} := \left\{y\in\Omega: P_f^\tau(y,\rho)<p_f^\tau(y)+\delta\right\}.
\end{equation}
Fix now $\delta>0$. Then we can find $\rho>0$ small enough such that $|Df|(\Omega\setminus E_{\rho,\delta})\le \delta$. Let us consider the set $A_{\rho,\delta}\subseteq E_{\rho,\delta}$ of $|Df|$-Lebesgue points of the function $p_f^\tau$, namely the points $x$ satisfying
\begin{equation}\label{eq:Lebesgue_upper_bound}
\lim_{r\to 0}\fint_{Q_r}|p_f^\tau(y)-p_f^\tau(x)|\,d|Df|(y)=0
\end{equation}
for every family of cubes $Q_r$ such that $\ell(Q_r)=r$, $x\in \tau Q_r$. By Theorem \ref{thm:Lebesgue_diff_tau_cubes_functions}, $|Df|$-almost every point has this property, hence
$|Df|(\Omega\setminus A_{\rho,\delta})\le \delta$. Moreover, by choosing some $\eps\in(0,\rho)$ small enough and using \eqref{eq:Lebesgue_upper_bound}, we can find a subset $C_{\rho,\delta}\subseteq A_{\rho,\delta}$ satisfying $|Df|(\Omega\setminus C_{\rho,\delta})\le 2\delta$ and such that for every $x\in C_{\rho,\delta}$ it holds
\begin{equation}\label{eq:Lebesgue_continuity_p_tau}
p_f^\tau(x)\le \fint_Q p_f^\tau(y)\,d|Df|(y)+\delta
\end{equation}
for every $Q$ with $\ell(Q)\le\eps$, $x\in \tau Q$.

By Lemma~\ref{prop:good_family} we can find a \emph{good family} $\mathcal{F}_\eps$ of disjoint cubes contained in $\Omega$: 
\begin{equation}\label{eq:good_family_upper_bound}
    \ell(Q)\leq \eps, \quad |Df|(\tau Q)\ge \frac{1}{16}|Df|(Q) 
\end{equation}
for every $Q\in \Fcal_\eps$, and
\begin{equation}\label{eq:almost_maximising}
\Gsf_\eps(f,\Omega)\leq \delta |Df|(\Omega)+ \sum_{Q \in \mathcal F_{\eps}}\Osc(f,Q)\ell(Q)^{n-1}.
\end{equation}

The next step is to split the last summand in~\eqref{eq:almost_maximising} into two further sums. One, over cubes $Q$ such that $\tau Q$ intersects $C_{\rho,\delta}$, and the second, as the sum over the remaining cubes. To this end, let us introduce the family 
\[
\Acal \coloneqq \{ Q \in \Fcal_\e : \tau Q\cap C_{\rho,\delta}\neq \emptyset\}\,.
\]

Let us first consider cubes $Q \in \Fcal_\e \setminus \Acal$. 
Poincar\'e inequality yields the bound
\begin{align}
\sum_{Q \in \mathcal F_{\eps} \setminus \Acal}  \Osc(f,Q)\ell(Q)^{n-1}  &\le \sum_{Q \in \mathcal F_{\eps} \setminus \Acal} 
\tfrac12 |Df|(Q) \nonumber \\ 
&\overset{\eqref{eq:good_family_upper_bound}}{\le} \sum_{Q \in \mathcal F_{\eps} \setminus \Acal} 
8 |Df|(\tau Q) \nonumber \\ 
&\le 8|Df|(\Omega\setminus C_{\rho,\delta})\le 16\delta.\label{eq:first_summand}
\end{align}

Let us now consider the case $Q \in \Acal$. By definition, for every such cube $Q$ we can find a point contained in $\tau Q\cap C_{\rho,\delta}$, which we call $\bar x_Q$, and which satisfies \eqref{eq:Lebesgue_continuity_p_tau}.
From this, we deduce the estimate
\begin{align} 
\sum_{Q \in \Acal}  \Osc(f,Q)\ell(Q)^{n-1} &= \sum_{Q \in \Acal}  \frac{\Osc(f,Q)\ell(Q)^{n-1}}{|Df|(Q)} |Df|(Q)\nonumber \\
&\le \sum_{Q \in \Acal} P^\tau_f(\bar x_Q, \rho) |Df|(Q)\nonumber \\
&\overset{\eqref{eq:E_rho_delta_definition}}{\le} \sum_{Q \in \Fcal_{\eps}} (p_f^\tau(\bar x_Q)+\delta)|Df|(Q) \nonumber \\
&\overset{\eqref{eq:Lebesgue_continuity_p_tau}}{\le}  \sum_{Q \in \Fcal_{\eps}} \Big(\fint_Q p_f^\tau(y)\,d|Df|(y)+2\delta\Big)|Df|(Q)\nonumber \\
&\le\int_\Omega p_f^\tau(y)\, d|Df|(y)+2\delta |Df|(\Omega).\label{eq:second_summand}  
\end{align}
Plugging the estimates \eqref{eq:first_summand}-\eqref{eq:second_summand} into~\eqref{eq:almost_maximising}, we obtain 
\begin{align}
\Gsf(f,\Omega)\le\Gsf_\eps(f,\Omega)& \leq 16\delta+3\delta |Df|(\Omega)+ \int_\Omega p_f^\tau(y)\, d|Df|(y).\label{eq:estimate_on_G_eps}
\end{align}
As $\delta$ was arbitrary, we deduce 
\[
\Gsf(f,\Omega)\le \int_\Omega p_f^\tau(y)\, d|Df|(y).
\]
Since this holds for every $\Omega$, an application of Radon-Nikodym implies that $g_f(x)\le p_f^\tau(x)$ for $|Df|$-almost every $x$. 
This finishes the proof.
\end{proof}

\begin{remark}[On the need of largely uncentered cubes]\label{rmk:limitation_of_tau_n} The restriction $\tau \ge \tau(n)$ is an essential component in our proof of the upper bound. The argument builds on the existence of a good family of cubes (cf. Section~\ref{sec:good_family}), which requires us to work with large values of $\tau < 1$. For $\tau = 1$, the proof of the upper bound does not rely on the use of good families, and thus the contents of Section~\ref{sec:good_family} are superfluous. In any case, we do not know the optimal value of $\tau(n)$. In principle, the value of $\tau(n)$ could be zero, which would correspond to taking centered cubes. 
\end{remark}

Combining Lemma \ref{lemma:g_ge_p_tau} and Lemma \ref{lemma:g_le_p_tau} we obtain the equality 
\[
g_f(x)=p_f^\tau(x)\quad \text{at $|Df|$-almost every point},
\] 
provided that $\tau\in[\tau(n),1)$, and thus Theorem~\ref{thm:p_tau_equals_g} is proven. We record a very simple consequence of this equality and Remark~\ref{rmk:bounds_g_f}.

\begin{corollary}\label{cor:range_p} Let $\tau\in [\tau(n),1)$, and let $f\in BV_\loc(\Omega)$. Then, 
    \begin{equation}\label{eq:range_p_body}
        \frac 14 \le p_f^\tau 
        \le p_f^1 \le \frac 12\qquad\text{$|Df|$-almost everywhere}.
    \end{equation}
\end{corollary}

\vspace{0.5cm}

\section{Representation in terms of tangents}\label{sec:tangents}

The goal of this section is to characterize $p_f^\tau$ in terms of the uncentered tangents of $f \in BV$, which we introduce next.

\subsection{Introducing oriented cubes} In this section we treat cubes as objects with a specific orientation. An \emph{orientation} of a cube $Q$ is defined as an $n$-tuple $\bbf = \{b_1, \dots, b_n\}$, where each $b_i$ is the center of a face and $\{b_1 - x_Q,\dots,b_n -x_Q\}$ (with $x_Q$ being the cube's center) defines a frame of orthogonal vectors (see Figure~\ref{fig:oriented}). With a slight abuse of notation compared to the previous sections, we will use $Q$ to denote an oriented cube $(Q, \bbf_Q)$. Now, for any oriented cube $Q$ and its orientation $\bbf_Q$, there exists a unique angle-preserving affine map $T_Q: \R^n \to \R^n$ that satisfies two key properties:
\begin{itemize}
    \item \emph{Realizes the cube as the image of the unit cube:} $T(Q_0) = Q$.
    \item \emph{Preserves orientation:} $T
(\frac 12e_i) = b_i$ for each canonical basis vector $e_i$.
\end{itemize}
Furthermore, applying an angle-preserving affine map $T$ to an oriented cube $Q$ with orientation $\{b_1,\dots,b_n\}$ naturally induces a new orientation $\{T b_1,\dots,T b_n\}$ on the transformed cube $T(Q)$.

It is worth noting that using oriented cubes primarily serves the purpose of defining a unique transformation map $T_Q$ for each cube. Crucially, all the relevant quantities we consider below, such as $\Osc(f,Q)$ and $|Df|(Q)$, do not depend on the orientation of $Q$ and are thus well-defined.

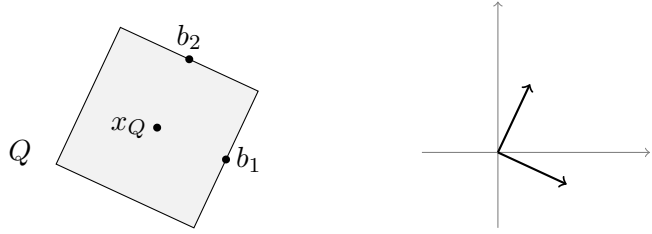
\begin{figure}
\begin{tikzpicture}[scale=2]
\begin{scope}[rotate=65]
		\draw[fill=gray!10] (0,0) -- ++(1,0) -- ++(0,1) -- ++(-1,0) -- cycle;
		\fill (0.5,0) circle (.75pt) node[right] {$b_1$};
		\fill (.5,.5) circle (.75pt) node[left] {$x_Q$};
		\fill (1,0.5) circle (.75pt) node[above] {$b_2$};
	\end{scope}
	\fill (-1,.5) circle (0pt) node[left] {$Q$};

        \draw[->,gray] (1.5,.5) -- (3,.5);
        \draw[->,gray] (2,0) -- (2,1.5);
        
        \begin{scope}[rotate around={-25:(2,0.5)}]
		\draw[->, thick]  (2,.5) -- (2,1);
		\draw[->, thick]  (2,.5) -- (2.5,.5);
        \end{scope}
	\end{tikzpicture}
 \caption{An oriented cube $Q\subset \R^2$. The vectors $b_1-x_Q$ and $b_2-x_Q$ are orthogonal.}\label{fig:oriented}
 \end{figure}
 
\subsection{Uncentered tangents} 
Given $f \in BV_\loc(\Omega)$, and given an (oriented) cube $Q$ such that $|Df|(Q)>0$, let us define the rescaling
$f_Q \colon Q_0 \to \R$ by setting 
\begin{equation}\label{eq:def_rescaled_1D}
    f_Q(y) \coloneqq  \frac{1}{|Df|(Q)} \, \left( f(T_Q y) - \fint_Q f \right)\, \ell(Q)^{n-1}.
\end{equation}
Notice that $f_Q$ has zero average and $|Df_Q|(Q_0)=1$, and
\begin{equation}\label{eq:push_forward}
    Df_Q = \frac{(T_Q^{-1})_\# Df}{|Df|(Q)}
\end{equation}
where $(T_Q^{-1})_\# : \Mcal(Q;\R^n) \to \Mcal(Q_0;\R^n)$ is the usual push-forward of measures (cf.~\cite[Def.~1.70]{AFP}).

\begin{definition}[Blow-up sequences and tangents]\label{def:tangents} Fix $\tau\in [0,1]$ and $f \in BV_\loc(\Omega)$. Let $x\in \spt(|Df|)$ and let $(Q_j)_j$ be a sequence of oriented cubes contained in $\Omega$. A sequence of rescaled functions $(f_{Q_j})_j \subset BV(Q_0)$ is called a \emph{$\tau$-centered blow-up sequence of $f$ at $x$} provided that
\[
\text{$x \in \tau Q_j$ \quad and \quad $\ell(Q_j) \to 0$ as $j \to 0$}\,.
\] 
Any strong $L^1(Q_0)$-limit of a $\tau$-centered blow-up sequence is called a \emph{$\tau$-centered tangent} (or simply \emph{$\tau$-tangent}) of $f$ at $x$. We denote the set of all $\tau$-tangents
of $f$ at $x$ by $\Tan^\tau(f,x)$. If $x\not\in \spt(|Df|)$ then we set $\Tan^\tau(f,x)=\emptyset$.
\end{definition} 

\begin{remark} By standard $BV$-compactness, $\Tan^\tau(f,x)$ is non-empty for all $x \in \spt(|Df|)$. 
    We shall see later by means of Proposition~\ref{prop:equality_c_sup_osc} that, additionally, $\Tan^\tau(f,x) \neq \{0\}$ for $|Df|$-almost every $x$ (cf. Corollary \ref{cor:tangent_non_trivial}). 
\end{remark}

In the following we will reserve the letter $u$ to denote tangents. By the lower semicontinuity of the total variation, every tangent $u\colon Q_0 \to \R$ is a $BV$ function on $Q_0$ with $|Du|(Q_0)\le 1$. It will be convenient to introduce a special subclass of tangents, the set of \emph{probability $\tau$-tangents}
\[
\Tan^\tau_1(f,x):=\{u\in \Tan^\tau(f,x):\, |Du|(Q_0)=1\}\,.
\]

Next, we record a couple of simple observations concerning blow-up sequences and $\tau$-tangents.

\begin{lemma}[Composition rule]\label{lemma:blow_up_composition}
    Let $Q$ be any cube in $\R^n$, and let $S$ be a cube contained in $Q_0$. Let $T_Q:Q_0\to Q$ be the affine map associated with $Q$. Then
    \begin{equation}\label{eq:blow_up_composition}
    (f_Q)_S=f_{T_Q(S)}.
    \end{equation}
\end{lemma}

\begin{proof}
    If $T_S:Q_0\to S$ is the affine map associated with $S$, then $T_{T_Q(S)}=T_Q \circ T_S$. Indeed, by definition
    \[
    T_Q\circ T_S (Q_0)=T_Q(S)=T_{T_Q(S)}(Q_0).
    \]
    Moreover, let $\bbf_Q=\{q_1,\ldots,q_n\}$ and $\bbf_S=\{s_1,\ldots,s_n\}$ be the orientations of $Q$ and $S$. Then
    \[
    T_Q\circ T_S (\tfrac12e_i)=T_Q(s_i)=T_{T_Q(S)}(\tfrac12 e_i).
    \]
    It follows that the maps $T_{T_Q(S)}$ and $T_Q\circ T_S$ coincide on the basis $\{e_i\}$, so they must coincide. Therefore,  $(f_Q)_S$ and $f_{T_Q(S)}$ have the same affine rescaling in the inner variable, have average zero and total variation one. This uniquely determines the function, so they must coincide.
\end{proof}

\begin{lemma}[Tangents are compact]\label{lemma:tangents_are_compact}
    Fix $\tau\in[0,1]$ and let $f\in BV_\loc(\Omega)$. Then, for every $x\in\Omega$, the set $\Tan^\tau(f,x)$ is compact (and thus closed) in the $L^1(Q_0)$ topology.
\end{lemma}

\begin{proof}
    We can clearly assume that $x\in\spt(|Df|)$. Take a sequence $u_j\in\Tan^\tau(f,x)$. By definition, $|Du_j|(Q_0)\le 1$ and $\fint_{Q_0} u_j=0$. By compactness in $BV$, up to subsequence $u_j $ converge in $L^1(Q_0)$ to some $v\in BV(Q_0)$ with $|Dv|\le 1$. For every $j$, let $(Q_j^i)_{i\in\mathbb{N}}$ be a sequence of cubes with $x\in \tau Q_j^i$ such that
    \[
    f_{Q_j^i}\to u_j\quad \text{in $L^1(Q_0)$ as $i\to\infty$.}
    \]
    Then, by a diagonal argument we can find a sequence $Q_j^{i(j)}$ such that
    \[
    f_{Q_j^{i(j)}}\to v\quad\text{in $L^1(Q_0)$ as $j\to\infty$.}
    \]
    This shows that $v\in \Tan^\tau(f,x)$ and concludes the proof.
\end{proof}

\subsection{Local Poincar\'e constants and tangents}
Let us now turn to the main result of this section, which states that $p_f^\tau(x)$ can be recovered by maximizing the oscillation of (probability) $\tau$-tangents at $x$. First, we define the quantities
\begin{align*}
    m^\tau(x)&:=\sup\left\{\frac{\Osc(u,Q_0)}{|Du|(Q_0)}:\, u\in \Tan^\tau(f,x),\, |Du|(Q_0)>0\right\}\\
    m^\tau_1(x)&:=\sup\,\{\Osc(u,Q_0):\,u\in \Tan_1^\tau(f,x)\}.
\end{align*}
We convene that the supremum over the emptyset is zero.

\begin{lemma}\label{lemma:m_m1_attained} Let $\tau \in [0,1)$ and let $f\in BV_\loc(\Omega)$. For every $x\in \Omega$ the following holds:
\begin{enumerate}[(i)]
    \item If $u\in \Tan^\tau(f,x)$ is not the zero function, then $|Du|(Q_0)^{-1}u$ also belongs to $\Tan^\tau(f,x)$, and hence to $\Tan^\tau_1(f,x)$. 
    \item $m^\tau(x)=m^\tau_1(x)$. Moreover, if $\Tan^\tau(f,x)\neq \{0\}$ then both suprema are attained. In addition, they also coincide with
    \[
    \sup\{\Osc(u,Q_0):\, u\in \Tan^\tau(f,x)\}.
    \]    
\end{enumerate}
\end{lemma}

\begin{proof}
    $(i)$ Let $Q_j$ be such that $x\in\tau Q_j$, $\ell(Q_j)\to 0$ and $f_{Q_j}\to u$ in $L^1(Q_0)$. 
    Define the points $y_j:=T_{Q_j}^{-1}(x)\in \tau Q_0$. Up to a (not relabeled) subsequence, we can assume that $y_j\to y_\infty\in \tau \overline{Q}_0$. For $\eta>0$, we also define the similarity
    \[
    D_\eta(y)\coloneqq y_\infty+\eta(y-y_\infty).
    \]
    For the sake of simplicity we write $Q_0^\eta\coloneqq D_\eta(Q_0)$ and $Q_j^\eta\coloneqq T_{Q_j}(Q_0^\eta)\subseteq Q_j$. Observe that for sufficiently large $j$ we still have $y_j\in \tau Q_0^\eta$, or equivalently, $x\in \tau Q_j^\eta$. Therefore $(Q_j^\eta)_j$ is an admissible sequence for $\Tan^\tau(f,x)$.

    For all $\delta>0$, we can always choose $\eta$ such that $1-\delta<\eta<1$ and $|Du|(\partial Q_0^\eta)=0$. In particular,
    \[
    |Du|(Q_0^\eta)=\lim_{j\to\infty}|Df_{Q_j}|(Q_0^\eta).
    \]
    By definition and the composition rule \eqref{eq:blow_up_composition},
    \[
    f_{Q_j^\eta}=(f_{Q_j})_{Q_0^\eta}=\frac{1}{|Df_{Q)_j}|(Q_0^\eta)}\left(f_{Q_j}\circ D_\eta-\fint_{Q_0^\eta} f_{Q_j}\right) \ell(Q_0^\eta)^{n-1}.
    \]
    Moreover, this sequence $L^1$-converges to 
    \[
    \frac{1}{|Du|(Q_0^\eta)}\left(u\circ D_\eta-\fint_{Q_0^\eta} u\right)\eta^{n-1}.
    \]
    Observe that, as $\eta\to 1$, this family $L^1$-converges to $|Du|(Q_0)^{-1} u$. Now if we fix a sequence $\eta_j\to 1$, we can extract a diagonal subsequence so that $f_{Q_j^{\eta_j}}$ converge to $|Du|(Q_0)^{-1} u$. This concludes the proof of $(i)$.

    $(ii)$ Since $\Tan^\tau(f,x)\neq\{0\}$ then $m^\tau(x)>0$. Indeed, the oscillation is zero if and only if the function is constant, and thus identically zero (because every tangent has zero average). Let now $u_j\in\Tan^\tau(f,x)$ be a maximizing sequence, i.e.,
    \[
    \frac{\Osc(u_j,Q_0)}{|Du_j|(Q_0)}\to m^\tau(x)>0.
    \]
    Then by the previous point, each $\tilde u_j:=|Du_j|^{-1}(Q_0)u_j$ belongs to $\Tan^\tau_1(f,x)$. By Lemma \ref{lemma:tangents_are_compact}, up to a (not relabeled) subsequence, $\tilde u_j\to v$ for some $v\in \Tan^\tau(f,x)$. First, we will show that $v\ne 0$. By Lemma \ref{lemma:almost_constant}, there exist constants $\eta\in (0,1)$ and $\delta>0$ such that $|D\tilde u_j|(\eta Q_0)\ge \delta|D\tilde u_j|(Q_0)=\delta$ for every $j$. This implies that $|Dv|(\overline{\eta Q_0})\ge \limsup_j |D\tilde u_j|(\overline{\eta Q_0})\ge \delta$. Now we use the continuity of the oscillation and the lower semicontinuity of the total variation on open sets to infer that
    \[
    \frac{\Osc(v,Q_0)}{|Dv|(Q_0)}\ge\limsup_j \frac{\Osc(\tilde u_j,Q_0)}{|D\tilde u_j|(Q_0)}=m^\tau(x).
    \]
    This shows that $v$ is a non-zero maximizer in $\Tan^\tau(f,x)$, with $|Dv|(Q_0)$; for otherwise there would be a strict inequality above, contradicting that $m^\tau(x)$ is the supremum. In particular, $v$ is also a maximizer in $\Tan^\tau_1(f,x)$. Since $\Tan^\tau_1(f,x)\subseteq \Tan^\tau(f,x)$, we conclude that $m^\tau(x)=m_1^\tau(x)$. Finally, the last claim in (ii) follows from the inequalities
    \[
    m_1^\tau(x)\le \sup\{\Osc(u,Q_0):\, u\in \Tan^\tau(f,x)\}\le m^\tau(x).
    \]
    This finishes the proof.
\end{proof}

We can now prove Theorem \ref{thm:p_tau_equal_osc} by means of the following result: 

\begin{proposition}\label{prop:equality_c_sup_osc}
Let $\tau \in [0,1)$ and let $f\in BV_\loc(\Omega)$. Then, for all $x \in \Omega$ it holds
    \begin{equation}\label{eq:equality_p_m_m1}
    p_f^\tau(x)=m^\tau(x)=m^\tau_1(x). 
    \end{equation}
\end{proposition}

\begin{proof} We divide the proof into steps consisting of different inequalities.

    \emph{Step 1.} $m^\tau(x)\ge p_f^\tau(x)$. 

    We may assume that $p_f^\tau(x)>0$. Let $Q_j$ be a sequence of cubes such that $x\in \tau Q_j$ and 
    \[
    \Osc(f_{Q_j},Q_0)=\frac{\Osc(f,Q_j)}{|Df|(Q_j)}\ell(Q_j)^{n-1}\to p_f^\tau(x),\qquad\text{as $j\to\infty$.}
    \]
    Since $|Df_{Q_j}|(Q_0)=1$, $f_{Q_j}$ is a pre-compact sequence in $L^1(Q_0)$ and converges (up to not relabeled subsequences) to some $u\in BV(Q_0)$ with $|Du|(Q_0)\le 1$. We claim that $|Du|(Q_0)>0$. Indeed 
    \[
    \Osc(u,Q_0)=\lim_{j\to\infty} \Osc(f_{Q_j},Q_0)=p_f^\tau(x)>0.
    \]
    Therefore, $u$ cannot be constant and hence $|Du|(Q_0)>0$. This proves that $\Tan^\tau(f,x)\ne \{0\}$. Moreover, this also shows that
    \[
    \frac{\Osc(u,Q_0)}{|Du|(Q_0)}\ge \Osc(u,Q_0)\ge p_f^\tau(x),
    \]
    whence $m^\tau(x)\ge p_f^\tau(x)$.

    \emph{Step 2.} $ p_f^\tau(x)\ge m^\tau_1(x)$.

    As before, we may assume that $m^\tau_1(x)>0$, and thus that $\Tan_1^\tau(f,x)\ne\emptyset$. Take any $v\in \Tan_1^\tau(f,x)$, and a generating blow-up sequence sequence $f_{Q_j}\to v$. Then
    \begin{align*}
    \Osc(v,Q_0) & =\lim_j \Osc(f_{Q_j},Q_0)\\
    & =\lim_j\frac{\Osc(f,Q_j)}{|Df|(Q_j)}\ell(Q_j)^{n-1}\\
    & \le \lim_j P^\tau(x,\ell(Q_j))= p_f^\tau(x).
    \end{align*}
    Taking the supremum among all $v\in\Tan_1^\tau(f,x)$ we obtain the sought assertion. 
    
    \emph{Step 3. Conclusion.} By Lemma \ref{lemma:m_m1_attained} and by the previous points, we have $p_f^\tau(x)\le m^\tau(x)= m^\tau_1(x)\le p_f^\tau(x)$, and therefore all inequalities are equalities. This completes the proof.
\end{proof}

From the assertion of the previous proposition and Theorem \ref{thm:p_tau_equals_g} we deduce the following: 

\begin{corollary}\label{cor:tangent_non_trivial}
    Let $\tau\in[0,1)$ and $f\in BV_\loc(\Omega)$. Then, for $|Df|$-almost every $x\in\Omega$, the tangent set $\Tan^\tau(f,x)$ contains some non-zero function.
\end{corollary}

\begin{proof}
    By Theorem \ref{thm:p_tau_equals_g} we know that, for $|Df|$-almost every $x$, $p_f^\tau(x)=g_f(x)$. Moreover,  $g_f(x)\in[\tfrac14,\tfrac12]$ for $|Df|$-almost every $x$. Proposition \ref{prop:equality_c_sup_osc} shows that $p_f^\tau(x)$ also coincides with the supremum of $\Osc(u,Q_0)$ among all tangents $u\in\Tan^\tau(f,x)$. It follows that there must exist at least one $\tau$-tangent for which $\Osc(u,Q_0)\ne 0$, and which therefore cannot be the zero function.
\end{proof}

\subsection{Tangents to tangents are tangents}

We now demonstrate that repeatedly applying the blow-up procedure is equivalent to a single blow-up operation (at least almost everywhere, in a sense to be precisely defined later).
This property plays a crucial role in our subsequent proof of Theorem \ref{thm:p_tau_equal_osc}. This type of stability result was first established by Preiss~\cite{preiss} for \emph{tangent measures}. A few key distinctions from Preiss’ original setting are worth noting before delving into technical details. First, in our case the tangent is only defined in the cube $Q_0$, while Preiss' tangents are defined in the whole space. Additionally, our definition of tangents employs uncentered rescalings, while Preiss used only centered rescalings. 
Nevertheless, the proof follows closely the original one, or rather the one that can be found in \cite[Theorem 14.16]{mattila}.

\begin{proposition}\label{prop:tangents_tangents}
    Let $f\in BV_\loc(\Omega)$ and let $\tau\in[0,1)$. At $|Df|$-almost every point $a\in \Omega$, every $u\in \Tan^\tau(f,a)$ has the following properties:
    \begin{enumerate}[(i)]
        \item $u_Q\in \Tan^\tau(f,a)$ for all subcubes $Q\subset Q_0$ with $\tau Q\cap \spt(|Du|)\neq \emptyset$;
        \item $\Tan^\tau(u,y)\subseteq\Tan^\tau(f,a)$ for every $y\in \spt(|Du|)$.
    \end{enumerate}
    
\end{proposition}

\begin{proof}\phantom{,}\\
\emph{(i)} By Lemma \ref{lemma:m_m1_attained}$(i)$ we can assume that $u\in\Tan^\tau_1(f,x)$, namely, that no mass is lost at the boundary along the sequence that generates $u$. Indeed the rescaling $u_Q$ is invariant with respect to scalar multiplication of $u$. 

We first prove the claim under the additional assumption that $|Du|(\partial Q)=0$. 
    For positive integers $k,m$, let $A_{k,m}$ be the set of all points $a\in\Omega$ for which there exists $u_a\in\Tan^\tau(f,a)$ and a cube $Q_a\subset Q_0$ satisfying $|Du_a|(\tau Q_a)>0$, $|Du_a|(\partial Q_a)=0$ and moreover
    \begin{equation}\label{eq:bigger_than_1/k}
    \|(u_a)_{Q_a}-f_{Q}\|_{L^1(Q_0)}>\frac{1}{k}
    \end{equation}
    for all cubes $Q$ with $a\in \tau Q$ and $0<\ell(Q)<1/m$. In order to reach the conclusion, it suffices to show that $|Df|(A_{k,m})=0$ for all $k,m$.

    Suppose by contradiction that $|Df|(A_{k,m})>0$ for some $k,m$. By the separability of $L^1(Q_0)$, we find a set $A\subset A_{k,m}$ with $|Df|(A)>0$ and such that
    \begin{equation}\label{eq:smaller_than_1/2k}
    \|(u_a)_{Q_a}-(u_b)_{Q_b}\|_{L^1(Q_0)}<1/(2k)
    \end{equation}
    for every $a,b\in A$. Now, let $a$ be a density point of $A$ for $|Df|$, i.e.,
    \begin{equation}\label{eq:lebesgue_for_A}
    1 = \lim_{r\to 0}\frac{|Df|(A\cap Q_r)}{|Df|(Q_r)}
    \end{equation}
    for every family of cubes $\{Q_r\}_r$ with $\ell(Q_r)=r$ and $a\in\tau Q_r$. In light of Theorem \ref{thm:tau_Lebesgue}, $|Df|$-almost every point $a\in A$ is a density point of $A$; we now show that the conclusion holds at all such points.

    Let $(Q_j)_j$ be a sequence of cubes such that $a \in \tau Q_j$, $\ell(Q_j)\to 0$ and 
    \[
    u_a=\lim_{j\to\infty} f_{Q_j}.
    \]
    Call $T_j:=T_{Q_j}$ the affine maps sending $Q_0$ to $Q_j$. We claim that 
    \begin{equation}\label{eq:intersect_A}
        T_j(\tau Q_a)\cap A\neq\emptyset \quad\text{ for $j$ large enough.}
    \end{equation}
    Suppose by contradiction that this is not the case. Then, up to extracting a (not relabeled) subsequence, it holds 
    \begin{equation}\label{eq:contradiction_assumption_Q_j}
    T_j(\tau Q_a)\cap A=\emptyset \quad\text{ for every $j$.}
    \end{equation}
    Since $\tau Q_a\cap\spt(|Du_a|)\neq\emptyset$ by assumption, we further get
    \[
    0<|Du_a|(\tau Q_a)\leq \liminf_{j\to\infty}|Df_{Q_j}|(\tau Q_a) \overset{\eqref{eq:push_forward}}{=} \liminf_{j\to\infty}\frac{|Df|(T_j(\tau Q_a))}{|Df|(Q_j)}.
    \]
    From this, we deduce that 
    \[
    T_{j}(\tau Q_a)\cap\spt(|Df|)=T_{j}(\tau Q_a\cap\spt(|Df_{Q_j}|))\neq \emptyset
    \]
    whenever $j$ is sufficiently large. It follows from \eqref{eq:contradiction_assumption_Q_j} that
    \begin{align*}
    \frac{|Df|(Q_j\cap A)}{|Df|(Q_j)} & = \frac{|Df|([Q_j\cap (T_j(\tau Q_a))^c \cap A)}{|Df|(Q_j)} \\
     & \le \frac{|Df|([Q_j\cap (T_j(\tau Q_a))^c)}{|Df|(Q_j)}\\ 
     & = 1-\frac{|Df|(T_j(\tau Q_a))}{|Df|(Q_j)} \\ 
     &\!\! \overset{\eqref{eq:push_forward}}{=} 1-\frac{|Df_{Q_j}|(\tau Q_a)}{|Df_{Q_j}|(Q_0)} = 1-|Df_{Q_j}|(\tau Q_a).
    \end{align*}
    Therefore, by the lower semicontinuity of the total variation,
    \begin{align*}
    1\overset{\eqref{eq:lebesgue_for_A}}{=}\lim_{j\to\infty}\frac{|Df|(Q_j\cap A)}{|Df|(Q_j)}&\leq 1-\liminf_{j\to\infty} |Df_{Q_j}|(\tau Q_a)\\
    &\leq  1-|Du_a|(\tau Q_a)<1,
    \end{align*}
    where we used the assumption $|Du_a|(\tau Q_a)>0$. We have reached a contradiction, thus proving the claim \eqref{eq:intersect_A}.

    In light of \eqref{eq:intersect_A}, we can choose $j$ such that $\ell(Q_j)<1/m$ and such that there exists a point $b\in A\cap T_j(\tau Q_a)$. By \eqref{eq:smaller_than_1/2k} we must have
    \begin{equation}\label{eq:1/2k}
    \|(u_a)_{Q_a}-(u_{b})_{Q_{b}}\|_{L^1(Q_0)}<1/(2k).
    \end{equation}
    By the composition rule \eqref{eq:blow_up_composition}, we obtain the identity 
    \[
    (u_a)_{Q_a}=\lim_{j\to\infty}  (f_{Q_j})_{Q_a}=\lim_{j\to\infty} f_{T_j(Q_a)},
    \]
    Here, to deduce the first equality we have used that $|Du_a|(\partial Q_a)=0$, the equality $|Df_{Q_j}|(Q_0) = |Du_a|(Q_0)  =1$, and Lemma~\ref{lemma:rel_closed}
    to infer that $|Du_a|(Q_a)=\lim_{j\to\infty}|Df_{Q_j}|(Q_a)$.
    This says that we can choose $j$ such that $\ell(Q_j)<1/m$ and
    \[
    \|(u_a)_{Q_a}-f_{T_j(Q_a)}\|_{L^1(Q_0)}<1/(2k).
    \]
    Applying the triangle inequality with this estimate and \eqref{eq:1/2k} gives
    \[
    \|(u_{b})_{Q_{b}}-f_{T_j(Q_a)}\|_{L^1}<1/k.
    \]
    However, given that $b \in \tau T_j(Q_a)$ and $\ell(T_j(Q_a))\le\ell(Q_j)<1/m$, this contradicts \eqref{eq:bigger_than_1/k} for the point $b$ using $Q=T_j(Q_a)$.  
    We have thus reached a contradiction, whence $|Df|(A_{k,m})=0$. This finishes the proof of (i) under the assumption that $|Du|(\partial Q)=0$.

    To prove the general case we argue by approximation with inner cubes: given any $Q\subset Q_0$, let $Q_k \subset Q$, $k\in\mathbb{N}$, be a sequence of cubes such that $Q_k\nearrow Q$ and $|Du|(\partial Q_k)=0$. Applying the first part of the proposition to each $u_{Q_k}$, we discover that $u_{Q_k}\in \Tan^\tau(f,x)$. Since $u_{Q_k}\to u_Q$ in $L^1(Q_0)$ as $k\to\infty$, the closedness of the tangents (cf. Lemma \ref{lemma:tangents_are_compact}) yields $u_Q\in\Tan^\tau(f,x)$.
    This proves $(i)$. 

    \emph{(ii)} This part follows directly from $(i)$ and the closedness property of $\Tan^\tau(f,x)$. This finishes the proof.
\end{proof}

\begin{remark}
    The previous Proposition holds also for $\tau=1$ when $n=1$. It is sufficient to apply the Lebesgue differentiation theorem with uncentered intervals, which is a consequence of Theorem \ref{thm:vitali_general}.
\end{remark}

\begin{corollary}\label{cor:c_tang_below_c_f}
    Let $f \in BV_\loc(\Omega)$ and let $\tau\in [0,1)$. For $|Df|$-almost every point $x \in \R^n$ the following holds: for all tangents $u \in \Tan^\tau(f,x)$ and all cubes $Q\subseteq Q_0$ satisfying $\tau Q\cap \spt(|Du|)\neq \emptyset$ it holds
    \begin{equation}\label{eq:bound_osc_p_tau}
        \Osc(u_Q,Q_0)=\frac{\Osc(u,Q)}{|Du|(Q)}\ell(Q)^{n-1} \le p_f^\tau(x).
    \end{equation}
    In particular, for $|Df|$-almost every $x \in \Omega$,
    \[
        p_u^\tau(y) \le p_f^\tau(x)  \qquad \text{for every $y \in \spt (|Du|)$}
    \] 
    for all $u \in \Tan^\tau(f,x)$.
\end{corollary}

\begin{proof}
    The proof follows directly from Proposition \ref{prop:tangents_tangents} and \ref{prop:equality_c_sup_osc}.
\end{proof}

Corollary \ref{cor:BMO} is a direct consequence of this result: 

\begin{proof}[Proof of Corollary \ref{cor:BMO}]
By Corollary \ref{cor:c_tang_below_c_f}, at $|Df|$-a.e. $x$, for every $\tau$-tangent $u\in\Tan^\tau(f,x)$ and every interval $I\subseteq (-\frac12,\frac12)$ it holds \[
\Osc(u,I)\le \frac{\Osc(u,I)}{|Du|(I)}\le p_f^\tau(x).
\]
Taking the supremum over all intervals, we obtain $\|u\|_{\mathrm{BMO}((-\frac12,\frac12))}\le p_f^\tau(x)$ for every $\tau$-tangent $u$. On the other hand, by Lemma \ref{lemma:m_m1_attained} and Proposition \ref{prop:equality_c_sup_osc} we can find a tangent $u$ realizing $p_f^\tau(x)=\Osc(u,(-\frac12,\frac12))\le \|u\|_{\mathrm{BMO}((-\frac12,\frac12))}$, and the claimed equality follows.
\end{proof}

\subsection{Tangents of \texorpdfstring{$SBV$}{SBV} functions}\label{sec:SBV_tangents}

In this section we compute explicitly the tangent sets $\Tan^\tau(f,x)$ associated with an $SBV$ function $f$. Given a hyperplane $H_{\nu,c}:=\{x\in\R^n:\,x\cdot\nu= c\}$, we define $j_{a,b,\nu,c}$ as the function that jumps from the value $a$ to the value $b$ when crossing the hyperplane $H_{\nu,c}$ in direction $\nu$ (see Figure~\ref{fig:tangents_SBV_jump}). 

\begin{center}
\begin{figure}
\begin{tikzpicture}[scale=1.7]

		
        \filldraw[fill=blue!30] (1.08,0) circle (0pt) node[right] {$j=a$}; 
  
		\filldraw[fill=red!30] (0,1.35) circle (0pt) node[below left] {$j = b$};

        
        \draw[line width=0.25mm](-1,-1) rectangle (1,1);

        \draw[fill=gray!20,dashed,line width=0mm] (-.75,-.75) rectangle (.75,.75);

        \draw[line width=0.25mm,color=black] (-1.6,-.2) -- (1.4,1.3) node[below right] {$H_{\nu,c}$};

        

        \filldraw (0,0) circle (0.5pt) node[below right] {};

        \filldraw (1,-1) circle (0pt) node[above right] {$Q_0$};

        \filldraw (.75,-.75) circle (0pt) node[above left] {$\tau Q_0$};

	\end{tikzpicture}
  \caption{The function $j=j_{a,b,\nu,c}$ of Lemma \ref{lemma:tangents_of_SBV} (in 2d). The function $j$ takes constant values $a,b$ and jumps from value $a$ to $b$ across the line $H_{\nu,c}$. }\label{fig:tangents_SBV_jump} 
 \end{figure}
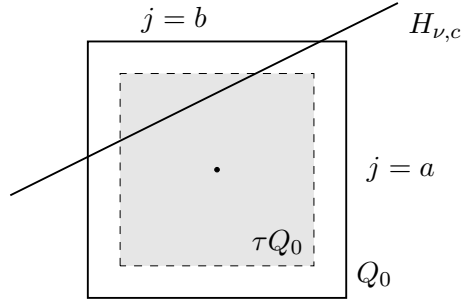
\end{center}
The following lemma is a direct consequence of the fine properties of $BV$ functions (see \cite{AFP}), and its proof is omitted: 
\begin{lemma}\label{lemma:tangents_of_SBV}
    Let $f\in SBV_\loc(\Omega)$, and $\tau\in [0,1)$. Then
    \begin{enumerate}[(i)]
        \item At $|D^af|$-almost every point, $\Tan^\tau(f,x)$ consists of all linear maps $u:Q_0\to\R$ with gradient of modulus 1;
        \item At $|D^jf|$-almost every point, $\Tan^\tau(f,x)$ consists of all the maps of the form $j_{a,b,\nu,c}$ among all $a,b,\in\R$, $\nu\in\mathbb{S}^1$ and $c\in \R$ such that $H_{\nu,c}\cap \tau \overline{Q_0}\ne\emptyset$ and 
        \[
        \int_{Q_0} j_{a,b,\nu,c}(x)\,dx  = 0, \qquad |Dj_{a,b,\nu,c}|(Q_0)=1.  
        \]
    \end{enumerate}
\end{lemma}

This result allows us to give a proof of Corollary \ref{cor:G_equals_K0_on_SBV}, i.e., 
\[
\Ksf_0(f)=\Gsf(f,\Omega)\,,
\quad f\in SBV_\loc(\Omega).
\]

\begin{proof}[Proof of Corollary \ref{cor:G_equals_K0_on_SBV}]
    By \eqref{eq:DFP} we know that $\Ksf_0(f)=\frac14 |D^af|(\Omega)+\frac12 |D^jf|(\Omega)$. Let us fix $\tau\in (0,1)$ sufficiently close to 1 so that Theorem \ref{thm:p_tau_equals_g} applies. We only need to verify that $p^\tau_f(x)$ is $\frac14$ at $|D^af|$-almost every point, and $\frac12$ at $|D^jf|$-almost every point. Recall from Theorem \ref{thm:p_tau_equal_osc} that $p^\tau_f(x)$ coincides with the largest oscillation among the $\tau$-tangents. Let us analyze absolutely continuous and jump points separately:

    For $|D^af|$-almost every point, tangents are linear with modulus 1, and by \cite[Lemma~3.1]{FMS} the supremum of $\Osc(u,Q_0)$ in this class is precisely $\frac14$. In fact, it is attained by functions whose gradient is aligned with one of the coordinate axes.

    For $|D^jf|$-almost every point, the jump function $u=j(-\frac12,\frac12,H_{e_1,0})$ (which is in $\Tan^\tau(f,x)$ by Lemma \ref{lemma:tangents_of_SBV}) satisfies $ \Osc(u,Q_0)=\frac12$. Since $p_f^\tau(x)\le \frac12$ always, we conclude $p^\tau_f(x)=\frac12$.
    
    The proof is complete.
\end{proof}

\vspace{0.5cm}

\section{Rigidity associated with extremal Poincaré constants}\label{sec:rigidity}

Motivated by the analysis in the $SBV$ case, one might hope to characterize  the jump points of $f\in BV$ as those where $p_f^\tau(x)=\frac12$, and the absolutely continuous points as those where $p_f^\tau(x)=\frac14$. Unfortunately, this conclusion is not always true, as there exist Cantor functions on $[0,1]$ that reach value $\frac14$ or $\frac12$ at $|Df|$-almost every point.
Nevertheless, when $p_f^\tau$
attains an extreme value, it is possible to infer certain rigidity properties on the set of tangents. More specifically, we establish (see Propositions~\ref{prop:rigidity_1/2} and~\ref{prop:rigidity_1/4}) the following rigidity properties of $p_f^\tau(x)$:
\begin{itemize}\itemsep=5pt
    \item The value $p_f^\tau(x) = \frac 12$ is attained if and only if there exists a jump tangent at $x$.
    \item On the other hand, $p_f^\tau(x) = \frac 14$ if and only if every tangent $u \in \Tan^\tau(f,x)$ satisfies $p_u^\tau = \frac 14$ at $|Du|$-almost every point. 
\end{itemize}

We begin by recording the following generalization of Hadwiger's result~\cite{Hadwiger}, about the rigidity of functions that attain the maximal Poincar\'e constant on the unit cube. 

\begin{lemma}[Maximizers of {Poincar\'e} inequality on the unit cube]\label{lemma:maximizers_poincare}
    Let $u\in BV(Q_0)$ attain the maximal Poincaré  constant, namely,
    \[
    \frac{
    \Osc(u,Q_0)}{|Du|(Q_0)}=\frac12.
    \]
    Then, up to addition and multiplication  by a constant, $u$ is the characteristic function of a half-cube of $Q_0$. \footnote{By a half-cube of $Q_0$, we refer to a set of the form
\[
    \{y\in Q_0:\, y\cdot e_i\ge 0\}\qquad\text{or}\qquad \{y\in Q_0:\, y\cdot e_i\le 0\},
\]
where $e_1,\dots,e_n$ is the canonical basis of $\R^n$.} 
\end{lemma}

\begin{proof}
    We know that $\Osc(v,Q_0)\le \frac12 |Dv|(Q_0)$ for every $v \in L^1(Q_0)$. Therefore, the functions that attain equality are precisely, up to multiplication by a constant, the maximizers of
    \[
    \max\{\Osc(v,Q_0):\, |Dv|(Q_0)\le 1\}.
    \]
    Observe that the map $v\mapsto \Osc(v,Q_0)$ is convex. Moreover, the set of zero-average functions $v \in L^1(Q_0)$ with $|Dv|(Q_0)\le 1$ is convex and compact in $L^1(Q_0)$. 
    Let thus $u$ be a maximizer, that we can suppose satisfies $|Du|(Q_0)=1$.
    
    Using the coarea formula we write $u$ as a convex combination of rescaled characteristic functions of its super- and sub-level sets. More precisely, setting $m\coloneqq \Lcal^n\text{-}\mathrm{essinf}_{Q_0}\,u$ and $M\coloneqq \Lcal^n\text{-}\mathrm{esssup}_{Q_0} u\, $, we observe that the zero-average condition entails
    $m\le 0$ and $M\ge 0$, so that
    \[
    u(x)=\int_0^M \1_{\{u>t\}}(x) \,dt-\int_{m}^0 \1_{\{u<t\}}(x)\, dt
    \]
    and
    \[
    |Du|(Q_0)=\int_0^M |D\1_{\{u>t\}}|(Q_0)\, dt+\int_{m}^{0} |D\1_{\{u<t\}}|(Q_0)\, dt.
    \]
    Define
    \[
    u_t(x):=
    \begin{cases}
        \frac{1}{|D\1_{\{u>t\}}|(Q_0)}\1_{\{u>t\}}(x) & \text{if $t\in(0,M)$}\\
        -\frac{1}{ |D\1_{\{u<t\}}|(Q_0)}\1_{\{u<t\}}(x) & \text{if $t\in(m,0)$}
    \end{cases}
    \]
    and 
    \[
    \lambda(t):=
    \begin{cases}
        |D\1_{\{u>t\}}|(Q_0) & \text{if $t\in(0,M)$}\\
        |D\1_{\{u<t\}}|(Q_0) & \text{if $t\in(m,0)$}
    \end{cases}\,.
    \]
    Notice that $\lambda(t)>0$ for $t\in (m,M)$, because in this range the sub- and super-level sets are non-trivial. Hence, the functions $u_t$ are well-defined, and moreover
    \[
    \int_m^M \lambda(t) \, dt = |Du|(Q_0) = 1.  
    \]
    By construction, $|Du_t|(Q_0)=1$ for every $t \in (m,M)$, and $u$ can be written as a convex combination of the $u_t$'s:
    \[
    u(x)=\int_{m}^M \lambda(t) \,u_t(x)\, dt\,.
    \]
    Since $u$ maximizes the oscillation, and by convexity of the latter, we get
    \[
    \frac12 =\Osc(u,Q_0)\le \int_{m}^{M}\lambda(t) \Osc(u_t,Q_0) \, dt \le \int_{m}^M \lambda(t)\frac12 |Du_t|(Q_0)\, dt =\frac12.
    \]
    This means that all inequalities are equalities. In particular, 
    \[
    \Osc(u_t,Q_0) = \frac 12 \quad \text{for $\Lcal^1$-almost every $t \in (m,M)$.} 
    \]
    In turn, this implies that $u_t$ maximizes the oscillation in $Q_0$ for $\Lcal^1$-almost every $t \in (m,M)$. Hadwiger~\cite{Hadwiger} proved that, among characteristic functions $v=\1_E$, those satisfying $\Osc(v,Q_0)=\frac12 |Dv|(Q_0)$ correspond to the case when $E$ is a half-cube. This shows that $\Lcal^1$-almost every super- or sub-level set (in the range $(m,M)$) of $u$ must be a half-cube. Now, by monotonicity of super- and sub-level sets, and by the zero-average condition, there exists a half-cube $H$ such that 
    \begin{align*}
    \1_{\{u >t\}} &= \1_H  \quad  \text{ for $\Lcal^1$-a.e. $t \in (0,M)$,} \\
    \1_{\{u <t\}} &= \1_{H^c}  \quad  \text{for $\Lcal^1$-a.e. $t \in (m,0)$}.
    \end{align*}
    The conclusion now follows.    
\end{proof}

As a direct consequence of this result, combined with Theorem \ref{thm:p_tau_equal_osc}, we obtain the following rigidity for tangents at points in the set $\{p_f^\tau = \frac 12\}$:

\begin{proposition}[Rigidity for $p^\tau_f=1/2$]\label{prop:rigidity_1/2}
    Let $f\in BV_\loc(\Omega)$ and let $\tau\in[0,1)$. Then for every $x \in \Omega$, the following conditions are equivalent:
    \begin{enumerate}[(i)]
        \item $\Tan^\tau(f,x)$ contains a (non-trivial) jump function across a half-cube.
        \item $p_f^\tau(x)=\frac12$.
    \end{enumerate}
\end{proposition}

We now turn to the  analysis of rigidity of tangents corresponding to points in the set $\{p_f^\tau = \frac 14\}$:

\begin{proposition}[Rigidity for $p^\tau_f=1/4$]\label{prop:rigidity_1/4}
    Let $f\in BV_\loc(\Omega)$ and let $\tau\in[0,1)$. Then at $|Df|$-almost every $x\in\{p_f^\tau =\frac14\}$, every $u\in \Tan^\tau(f,x)$ satisfies
    \[
    \Osc(u,Q)\ell(Q)^{n-1}\le \frac14 |Du|(Q)
    \]
    for every cube $Q\subset Q_0$ such that $\tau Q\cap \spt |Du|\ne \emptyset$.
    In particular,
   \[
    p_u^\tau = \frac14\quad\text{$|Du|$-almost everywhere.}
    \]
\end{proposition}

\begin{proof}
    Without loss of generality, we can restrict to the set of points $x \in \Omega$ where tangents to tangents are tangents, because this set has full measure by Proposition \ref{prop:tangents_tangents}.
    Fix such an $x \in \{p_f^\tau =\frac14\}$. Then, Corollary~\ref{cor:c_tang_below_c_f} tells us that every  $u\in\Tan^\tau(f,x)$ satisfies the first assertion of the proposition, and also
    \[
    p_u^\tau\le \frac14\quad \text{$|Du|$-almost everywhere.}
    \]
    Since it always holds $p_u^\tau\ge \frac14$ $|Du|$-almost everywhere, the proof is complete.\qedhere 
\end{proof}

\vspace{0.5cm}

\section{A few interesting questions}  Let us now delve into a few intriguing questions about the structure of a function, its Poincar\'e constant and its $\tau$-tangents.
For the rest of this section, we fix $\tau \in [0,1)$ and let $f \in BV_\loc(\Omega)$. 

\subsection{Mono-directionality}
A standard result states that, at small scales around most points, $BV$ functions are $L^1$-close to a one-directional and monotone function. The following result formalizes this in the context of $\tau$-tangents.

\begin{proposition}\label{prop:monodirectional}
    Let $\tau\in[0,1)$. Then 
    for $|Df|$-a.e. point $x\in\Omega$, every tangent $u\in\Tan^\tau(f,x)$ coincides with the restriction to the unit cube $Q_0$ of a function of the form $h(x\cdot e)$ for some $e\in\R^n$ and some monotone function $h:\R\to\R\cup\{\pm\infty\}$.
\end{proposition}

\begin{proof}
    The proof follows the same lines of \cite[Lemma 4]{ARBDN22}. Indeed, by the uncentered differentiation theorem \ref{thm:tau_Lebesgue}, one obtains that, at $|Df|$-almost all points $x$, all $\tau$-tangents have constant polar. By approximation with smooth functions one obtains the conclusion with $e=\frac{Df}{|Df|}(x)$.
\end{proof}

It is natural to ask if there exist optimal tangents (realizing the Poincar\'e constant in the cell-formula maximization problem appearing in Theorem \ref{thm:p_tau_equal_osc}), whose mono-directionality is  aligned with one of the sides of $Q_0$:

\begin{question}[Alignment of optimal tangents]
    Is it true that for all (or some) tangents $u\in\Tan^\tau(f,x)$ realizing $p_f^\tau(x)$, we have $u(x)=h(x\cdot e)$ for some monotone function $h:\R\to\R$ and where $e$ is one of the coordinate directions? 
\end{question}

A positive answer to this question would indicate the possibility of reducing the proofs presented in this work to the one-dimensional case. This would create a viable pathway for proving Theorem~\ref{thm:p_tau_equals_g} for $\tau=1$ in all dimensions $n>1$. 

\subsection{Poincar\'e constants of tangents} 
We have observed that $\Tan^\tau(f,x)$ is naturally a compact set in $L^1(Q_0)$. We have also discussed that tangents to tangents qualify as tangents themselves. Additionally, Poincaré constants are always confined to take values in the compact interval $[\frac 14,\frac 12]$. Furthermore, the inequality $p_u^\tau \le p_f^\tau(x)$ holds for all tangents $u \in \Tan^\tau(f,x)$, suggesting the possibility of minimizing $p_u^\tau$ within $\Tan(f,x)$. This raises the prospect of finding at least one tangent with minimal and constant local Poincaré constants everywhere: 

\begin{question}\label{Q:minimization} 
    For $|Df|$-almost every $x \in \Omega$, does $\Tan^\tau(f,x)$ contain a tangent $\bar u$ attaining the minimum value 
    \[
          \alpha(x):=\inf_{u \in \Tan^\tau(f,x)} \int_{Q_0} p_u^\tau \, d|Du| = \inf_{u \in \Tan^\tau(f,x)} \Gsf(u,Q_0)\,,
    \]
    and, if so, is $p^\tau_{\bar{u}}$ constant  $|D\bar u|$-almost everywhere?
\end{question}

The key challenge in proving this claim lies in the lack of lower semicontinuity of the functional $\Gsf(\frarg,Q_0)$ with respect to $L^1$-convergence (or weak*-convergence in the $BV$ space for that matter).

The $|Df|$-approximate continuity of $p_f^\tau$ implies that, at $|Df|$-almost every point $x$, $p_f^\tau(x)$ is very close to being constant on small neighborhoods. Since tangents look at the infinitesimal behavior of $f$, this may suggest the existence of a tangent at $x$ whose local Poincar\'e constant is constantly equal to $p_f^\tau(x)$. With this in mind, we formulate the following question:

\begin{question}\label{Q:1} 
    Is it true that, at $|Df|$-almost every point $x\in \Omega$, there exists $u\in \Tan^\tau(f,x)$ such that $p_u^\tau = p_f^\tau(x)$ almost everywhere with respect to $|Du|$?
\end{question}
The answer to this question is affirmative for all points $x$ such that $p^\tau_f(x)\in\{\frac14,\frac12\}$ (cf. Propositions~\ref{prop:rigidity_1/2} and~\ref{prop:rigidity_1/4}). 
In addition, we observe that if one were able to prove that 
\begin{equation}\label{eq:claim_max_G_equal_p}
p^\tau_f(x) = \max_{u \in \Tan^\tau(f,x)}\Gsf(u,Q_0)\,, 
\end{equation}
then Question~\ref{Q:1} would have a positive resolution. Indeed, let $\bar u \in \Tan^\tau(f,x)$ be a maximizer, namely assume that 
\[
p^\tau_f(x) = \max_{u \in \Tan^\tau(f,x)}\Gsf(u,Q_0) = \Gsf(\bar u,Q_0) \,.
\]
Then, by Corollary~\ref{cor:c_tang_below_c_f} 
\begin{align*}
\int_{Q_0} p^\tau_{\bar u}(y) \, d|D\bar u|(y) & \le \int_{Q_0} p^\tau_{f}(x) \, d|D\bar u|(y) \\
& \le p^\tau_{f}(x) =\Gsf(\bar u,Q_0) = \int_{Q_0} p^\tau_{\bar u}(y) \, d|D\bar u|(y)\,.
\end{align*}
It follows that all inequalities are equalities, and thus $p_{\bar u}^\tau(y) = p_f^\tau(x)$ for $|D\bar u|$-almost every point.

At the moment, we are not able to prove~\eqref{eq:claim_max_G_equal_p}. However, we can establish the following weaker result: 
\begin{proposition}
    For $|Df|$-almost every $x$ it holds
    \[
        p_f^\tau(x) = \sup_{u\in\Tan^\tau(f,x)} \Gsf_1(u,Q_0)\,.
    \]
\end{proposition}
\begin{proof}
    Fix $u \in \Tan^\tau(f,x)$. By the  definition of $\Gsf_1$, we have 
    \[
    \Gsf_1(u,Q_0) \ge \Osc(u,Q_0)
    \]
    because the unit cube belongs to $\Hcal_{\le 1}(Q_0)$. Taking the supremum over all tangents and applying Theorem~\ref{thm:p_tau_equal_osc}, we deduce the inequality
    $$
    \sup_{u\in\Tan^\tau(f,x)} \Gsf_1(u,Q_0) \ge p_f^\tau(x)\,.
    $$
    The other inequality will follow  
    from Proposition~\ref{prop:good_family} and  Corollary~\ref{cor:c_tang_below_c_f}. Let $u \in \Tan^\tau(f,x)$ and let $\delta>0$. 
    We apply Proposition~\ref{prop:good_family} to find a good family $\Fcal \in \Hcal_{\le 1}(Q_0)$ for $\Gsf(u,Q_0)$. Notice that, for any cube $Q$, if $|Du|(Q)=0$ then $\Osc(u,Q)=0$. Therefore we can assume without loss of generality that  $|Du|(Q)>0$ for all cubes $Q \in \Fcal$. 
    Further, by the definition of good family, for every $Q \in \Fcal$ it holds
    \[
    |Du|(\tau Q) \ge \frac{1}{16}|Du|(Q)>0\,.
    \]
    This means that $\spt(|Du|) \cap \tau Q \ne \emptyset$ for all $Q \in \Fcal$, which
    allows us to apply Corollary~\ref{cor:c_tang_below_c_f} 
    for each cube of $\Fcal$: 
    \begin{align*}
        (1-\delta)\Gsf_1(u,Q_0) & \le \sum_{Q \in \Fcal} \Osc(u,Q)\ell(Q)^{n-1} \\ 
        & \!\!\!\overset{\eqref{eq:bound_osc_p_tau}}{\le} \sum_{Q \in \Fcal} p_f^\tau(x) |Du|(Q) \le p_f^\tau(x)|Du|(Q_0) \le p_f^\tau(x).
    \end{align*}
   Letting $\delta \downarrow 0$ first, and then taking the supremum over all tangents, we obtain the other inequality and we conclude the proof.
\end{proof}

Appealing to the same line of thought of the previous proof, one can show that the following inequalities hold at $|Df|$-almost every point: for every $\eps\in (0,1)$ and every $u\in\Tan^\tau(f,x)$
\[
\Gsf(u,Q_0)\le \Gsf_\eps(u,Q_0)\le \Gsf_1(u,Q_0)\le p_f^\tau(x).
\]
Motivated by Question~\ref{Q:minimization} and Question~\ref{Q:1}, one may also wonder whether the following weaker conclusion holds:
\begin{question}\label{question:p_u_constant}
    Is it true that, at $|Df|$-almost every point $x\in \Omega$, there exists $u\in \Tan^\tau(f,x)$ such that $p_u^\tau$ is constant almost everywhere with respect to $|Du|$?
\end{question}

While a conclusive answer remains elusive at present, we are able to establish the following result, which gives a partial answer to Question \ref{question:p_u_constant}:

\begin{proposition}[Almost minimizing tangents]\label{prop:almost_constant_inf} 
For $x \in \Omega$, let 
\begin{equation}\label{eq:def_beta}
\beta(x) :=\inf \{|Du|\text{-}\mathrm{essinf}\, p_u^\tau: \, u\in\Tan^\tau(f,x)\}\,.
\end{equation}
For $|Df|$-almost every point $x \in \Omega$, it holds 
\[
\frac 14 \le \beta(x) \le \frac 12\,.
\]
Moreover, at such points, 
for each $\delta >0$ there exists a non-trivial tangent $v \in \Tan^\tau(f,x)$ satisfying
\[
 \beta(x) \le p_v^\tau(y) \le \beta(x) + \delta
\]
for $|Dv|$-almost every $y \in Q_0$. 
    \end{proposition}
    \begin{proof} Let $A\subset \spt(|Df|)$ be the set of points $x$ where tangents to  tangents are tangents, Corollary \ref{cor:c_tang_below_c_f} applies, and $\Tan^\tau(f,x) \neq \{0\}$, which has full $|Df|$ measure (cf.  Proposition~\ref{prop:tangents_tangents} and Corollary~\ref{cor:tangent_non_trivial}). 
     Observe that, thanks to Corollary~\ref{cor:range_p}, we deduce that $\beta(x) \in [\frac14,\frac 12]$ for every $x \in A$. This proves the first assertion.

    For $x \in A$ and $\delta > 0$, we may hence find a non-trivial tangent  $u\in\Tan^\tau(f,x)$ satisfying 
    \[
    |Du|\text{-}\mathrm{essinf}\, p_u^\tau <\beta(x)+\delta.
    \]
    Choose $y \in \spt(|Du|)$ such that $\Tan^\tau(u,y) \neq \{0\}$ and  $p_u^\tau(y)\le \beta(x)+\delta$. Now, let $v\in\Tan^\tau(u,y)\subset \Tan^\tau(f,x)$ be a non-trivial tangent. Then, by Corollary \ref{cor:c_tang_below_c_f} we have $p_v^\tau\le p_u^\tau(y)\le\beta(x)+\delta$ at $|Dv|$-almost all points. On the other hand, by the definition of $\beta(x)$, we have $p_v^\tau\ge \beta(x)$ at $|Dv|$-almost all points. 
    This proves the second assertion.
    \end{proof}

    \begin{remark} If the minimization problem posed in Question~\ref{Q:minimization} has a minimizer $\bar u$ in $\Tan^\tau_1(f,x)$, then 
    $\alpha(x)=\beta(x)$, where $\beta(x)$ is defined in \eqref{eq:def_beta}. Indeed, let $\bar u \in \Tan^\tau_1(f,x)$ be the minimizer for $\alpha(x)$ and let $\tilde{u}$ be a tangent given by Proposition \ref{prop:almost_constant_inf}. Then, by minimality 
    \begin{align*}
    \alpha(x) & = \Gsf(\bar u, Q_0) \le \Gsf(\tilde u,Q_0) = \int_{Q_0} p_{\tilde{u}}^\tau \, d|D\tilde{u}| \le \beta(x) + \delta.
    \end{align*}
    On the other hand, by the definition of $\beta(x)$, it holds $\beta(x)\le p_{\bar u}^\tau$ almost everywhere with respect to $|D\bar u|$, hence by integration 
    \[
    \beta(x) \le \int_{Q_0} p_{\bar u}^\tau\, d|D\bar{u}| = \Gsf(\bar u, Q_0) = \alpha (x).
    \]
    Sending $\delta \to 0$ we conclude that $\alpha(x)=\beta(x)$, as we wanted. Moreover, in this case it follows that $p_{\bar u}^\tau$ is constant and equal to $\alpha(x)$ almost everywhere with respect to $|D\bar u|$, thus answering Question \ref{question:p_u_constant}. 
    \end{remark}

\subsection{Rigidity at points with minimal Poincar\'e constant} As we have already seen in Propositions \ref{prop:rigidity_1/2} and \ref{prop:rigidity_1/4}, there are certain additional properties of $\Tan^\tau(f,x)$, when $p_f^\tau(x)$ attains an extremal value. We also know that affine functions are a prototype of a mono-directional monotone function with local Poincar\'e constant equal to $\frac 14$ almost everywhere. In fact, if $u : Q_0 \to \R$ is linear, then 
\[
    \frac{\Osc(u,Q)}{|Du|(Q)}\,\ell(Q)^{n-1} = \frac 14
\]
for all subcubes $Q \subset Q_0$. In this context, the following question probes the validity of a converse statement, specifically exploring whether the upper bound $\frac 14$ on such quotients enforces affinity:
\begin{question}\label{Q:1/4}
    Let $u : Q_0 \to \R$ be a mono-directional and monotone $BV$ function. Does the inequality
    \[
    \Osc(u,Q)\ell(Q)^{n-1}\le \frac14 |Du|(Q),
    \]
    for all cubes $Q \subset Q_0$, 
    imply that $u$ is affine? If not, could we reach this conclusion by imposing the stricter equality condition 
    \[
    \Osc(u,Q)\ell(Q)^{n-1} = \frac14 |Du|(Q)
    \]
    for all cubes $Q \subset Q_0$?
\end{question}

\begin{remark}
    Without the monotonicity assumption, the answer to the first part of Question \ref{Q:1/4} is negative: a direct computation shows that the function $u(x)=|x|$, defined on $(-\frac12,\frac12)$, satisfies $\Osc(u,I)\le \frac14 |Du|(I)$ for every interval $I\subseteq (-\frac12,\frac12)$. 
\end{remark}

Proposition \ref{prop:rigidity_1/4} underscores the significance of this question: A positive resolution of either claim could lead to a positive answer to the following questions:

\begin{question}\label{question:B_implies_D} 
    Is it true that for $|Df|$-almost every $x \in \{p_f^\tau =\frac14\}$, there \emph{exists} a linear tangent in $\Tan^\tau(f,x)$?
\end{question}

\begin{question}\label{question:B_implies_A}  
Is it true that for $|Df|$-almost every $x \in \{p_f^\tau =\frac14\}$, \emph{every} tangent in $\Tan^\tau(f,x)$ is linear?
\end{question}

Lastly, we explore a potential converse implication in Proposition~\ref{prop:rigidity_1/4}.

\begin{question}\label{question:C_implies_B}
    Does $|Df|$-almost every point $x \in \Omega$ have the following property: if
     \[
     p_u^\tau= \frac14 \quad \text{$|Du|$-almost everywhere}
     \]
     for every tangent $u\in \Tan^\tau(f,x)$, then $p_f^\tau(x)=\frac14$? 
\end{question}

\subsection{Size of points with minimal Poincar\'e constant} 
By standard measure-theoretic arguments concerning the \emph{invariance directions} of tangent measures (see~\cite{Adolfo_PAMS}; see also~\cite{Ambrosio_Soner,Fragala_Mantegazza}), the following dimensional estimate holds: the set of points where all tangents of $f$ are linear has full dimension (provided it has positive $|Df|$-measure). This suggests the following question, 
which would have a positive answer in case Question \ref{question:B_implies_A} is solved affirmatively:

\begin{question}\label{question:last}
    Provided that it has positive $|Df|$-measure, does the set
    $\{p_f^\tau=\tfrac14\}$ 
    have Hausdorff dimension $n$?
\end{question}

\appendix
\vspace{0.5cm}
\section{Some measure theoretic results}
\label{app:measure_theory}
\addtocontents{toc}{\SkipTocEntry}
\subsection{A covering Theorem in \texorpdfstring{$\R$}{R}}

We recall some definitions. Let $A\subset \R$. A \emph{cover} of $A$ is any family of sets $\Fcal$ such that for every $x\in A$ there exists $U\in\Fcal$ with $x\in U$. We call $\Fcal(x)$ the collection of sets from $\Fcal$ that contain $x$.
We say that $\Fcal$ is a \emph{fine cover} of $A$ if for every $x\in A$ and for every $\delta>0$ there exists $U\in\Fcal(x)$ with $\mathrm{diam}(U)<\delta$.

\begin{definition}[Vitali property]\label{def:Vitali}
    Given a cover $\Fcal$ of $A$, and a Radon measure $\mu$, we say that $\Fcal$ has the $\mu$-\emph{Vitali property} if there exists a disjoint collection $\Fcal'\subset\Fcal$ such that
\[
\mu\left(A\setminus\bigcup_{U\in\Fcal'} U\right)=0.
\]
\end{definition}
It is well-known (see e.g. \cite[Theorem 2.2]{mattila}) that any fine cover $\Fcal$ of $A$ with closed balls (also uncentered ones) has the $\Lcal^n$-Vitali property. In the one dimensional case, relying on the natural order relation on the real line, one can actually show the same for every measure $\mu$. This observation can be found in  \cite[Chapter 1, Remark 5]{deguzman} along with some hints for its proof. For completeness, we report here a detailed argument.

\begin{theorem}[Vitali with closed intervals in $\R$]\label{thm:vitali_general}
Let $\mu$ be any Radon measure on the real line $\R$. Then every fine cover $\Fcal$ of a set $E$ with closed intervals has the $\mu$-Vitali property.
\end{theorem}

\begin{proof}
We can assume without loss of generality that $E\subseteq (0,1)$, and that every element of the cover is contained in $(0,1)$. Moreover, we can assume that $E\subseteq \spt(\mu)$, and therefore that $\mu(I)>0$ for every $I\in \Fcal$.

We start by selecting $S_1\in \Fcal$ such that
\[
\mu(S_1)\ge \frac12 \sup\{\mu(I):\, I\in \Fcal\}.
\]
Then we inductively select $S_k$, $k\ge 2$, so that
\[
S_k\cap\bigcup_{i=1}^{k-1} S_i=\emptyset
\]
and
\[
\mu(S_k)\ge \frac12 \sup\left\{\mu(I):\, I\in\Fcal, \, I\cap\bigcup_{i=1}^{k-1} S_i=\emptyset\right\}.
\]
If the process terminates in a finite number of steps then we are done, as $\mu$-almost all the set $E$ must be covered. Suppose then that the process goes on, giving an infinite sequence $S_k$, $k\ge 1$. Since the $S_k$'s are disjoint and $(0,1)$ has finite $\mu$-measure, we must have 
\begin{equation}\label{eq:S_k_to_0}
\mu(S_k)\to 0 \qquad\text{as $k\to\infty$.}
\end{equation}
We now claim that 
\begin{equation}\label{eq:I_empty_intersection}
    I\cap \bigcup_{k=1}^\infty S_k\neq \emptyset\qquad\text{for every $I\in \Fcal$}.
\end{equation}
Indeed, otherwise there would exist $I\in\Fcal$ disjoint from all the $S_k$, but then it must have been selected at some point in the process because of \eqref{eq:S_k_to_0} (recall that $\mu(I)>0$ for every $I\in \Fcal$). Therefore \eqref{eq:I_empty_intersection} is proven.

Now we prove that 
\[
\mu\left(E\setminus \bigcup_{k=1}^\infty S_k\right)=0.
\]
For every $k\ge 1$, let us define the \emph{enlarged} set 
\[
\widehat S_k:=\bigcup \{S\in\Fcal:\, \mu(S)\le 2\mu(S_k),\, S\cap S_k\neq\emptyset\}.
\]
We have the following chain of inclusions:
\begin{align}
\begin{aligned}\label{eq:S_k_chain}
    E\setminus \bigcup_{k=1}^h S_k& \subseteq \bigcup \{S\in \Fcal:\, S\cap \bigcup_{k=1}^hS_k =\emptyset\}\\
    & = \bigcup \{S\in \Fcal:\, S\cap \bigcup_{k=1}^h S_k =\emptyset,\, S\cap \bigcup_{k=h+1}^\infty S_k\neq\emptyset \}\\
    & =\bigcup_{j=h}^\infty \{S\in\Fcal:\, \bigcup_{k=1}^j S_k =\emptyset,\, S\cap S_{j+1}\neq\emptyset \}\\
    &\subseteq \bigcup_{j=h}^\infty \widehat S_{j+1}.
\end{aligned}
\end{align}
The final step is to show that 
\begin{equation}\label{eq:5mu(S_k)}
    \mu(\widehat S_k)\le 5 \mu(S_k).
\end{equation}
This is sufficient to conclude in view of \eqref{eq:S_k_chain}, and sending $h\to\infty$. 

Now \eqref{eq:5mu(S_k)} follows from the following inequality:
\[
\mu(\widehat I)\le \mu(I)+\sup\{\mu(I_L)+\mu(I_R):\,I_L,I_R\cap I\neq \emptyset,\, \mu(I_L),\mu(I_R)\le 2\mu(I)\}.
\]
Indeed, it is sufficient to choose $I_L$ and $I_R$ as the biggest intervals extending to the left and to the right of $I$ (and in case there is no biggest, one can choose a ``maximizing'' sequence and use inner regularity of $\mu$).\end{proof}

\begin{theorem}\label{thm:tau_vitali} Let $E\subset \Omega \subset \R^n$, with $\Omega$ open set, let $\mu^*: \Pcal(\R^n)\to [0,\infty]$ be an outer Radon measure on $\R^n$ and $\tau \in [0,1)$. Moreover, let $\Fcal$ be a cover of $E$ with closed cubes such that for every $x \in E$ there exists arbitrarily small cubes $Q \in \Fcal$ with $x \in \tau Q$.   
Then $\Fcal$ has the $\mu^*$-Vitali property, i.e., there exists a countable family $\Fcal_0 \subset \Fcal$ of cubes with pairwise disjoint closures such that 
\[
\mu^{*} \bigg(E \setminus \bigcup_{Q \in \Fcal_0} \overline{Q}\bigg) = 0.
\]
\end{theorem}

\begin{proof}
We can apply \cite[Thm. 1.147]{Fonseca-Leoni}. Indeed, the cover $\Fcal$ is a fine $\gamma$-Morse cover (see \cite[Def. 1.137]{Fonseca-Leoni}), with $\gamma=\frac{2\sqrt{n}}{1-\tau}$. Indeed, if $x\in \tau Q$ then by elementary considerations $\overline{B}(x,\frac{1-\tau}{2}\ell(Q)) \subset \overline{Q} \subset \overline{B}(x,\sqrt{n} \ell(Q))$. 
\end{proof}

We now recall the following Lebesgue-differentiation type results. Notice the set $A$ in the next statement need not be measurable (see also \cite[Corollary 2.14(1) and Remark 2.15(2)]{mattila}).

\begin{theorem}\label{thm:tau_Lebesgue}
    Let $\Omega\subset \R^n$, and $\mu$ a Radon non-negative measure on $\Omega$. Fix $\tau\in[0,1)$ and a set $A\subset \Omega$. 
    Then for $\mu$-almost every $a\in A$ the following happens: for every family of cubes $(Q_r)_{r>0}$, with $a\in\tau Q_r$ and $\ell(Q_r)\to 0$ as $r\to0$, 
    \[
    \lim_{r\to0} \frac{\mu(A\cap Q_r)}{\mu(Q_r)}=1. 
    \]
\end{theorem}

\begin{theorem}\label{thm:Lebesgue_diff_tau_cubes_functions} Let $\Omega \subset \R^n$ and $\mu$ a Radon non-negative measure on $\Omega$. Let $u$ be a locally $\mu$-integrable function. Then for $\mu$-a.e. $x \in \Omega$ the following happens: for every family of cubes $(Q_r)_{r>0}$, with $x\in\tau Q_r$ and $\ell(Q_r)\to 0$ as $r\to0$, 
    \[
    \lim_{r\to0} \fint_{Q_r} |u(y)-u(x)| \, d\mu(y) = 0  .
    \]
\end{theorem}
\addtocontents{toc}{\SkipTocEntry}
\subsection{A criterion for set-functionals} Let $\Omega$ be an open set in $\R^n$, and let $\Ocal(\Omega)$ be the family of all open subsets of $\Omega$.
\begin{theorem}[De Giorgi-Letta {\cite[Theorem 1.53]{AFP}}]\label{thm:De_Giorgi_Letta}
Let $\Fsf:\Ocal(\Omega)\to[0,\infty]$ be an increasing set function with $\Fsf(\emptyset)=0$. Assume that the following properties hold:
\begin{enumerate}[(i)]
    \item If $U,V\in\Ocal(\Omega)$ then $\Fsf(U\cup V)\leq \Fsf(U)+\Fsf(V)$ {\rm (subadditivity)};
    \item If $U,V\in\Ocal(\Omega)$ and $U\cap V=\emptyset$ then $\Fsf(U\cup V)\geq \Fsf(U)+\Fsf(V)$ {\rm (superadditivity)};
    \item $\Fsf(U)=\sup\{\Fsf(V): \,V\in\Ocal(\Omega),\, V\subset\subset U\}$ {\rm (inner regularity)}.
\end{enumerate}
Then there is a uniquely determined measure $\mu$ that extends $\Fsf$ to all Borel subsets, given by
\[
\mu(B):=\inf \{\Fsf(U):\,U\in\Ocal(\Omega),\, U\supset B\},\qquad\text{$B$ Borel}.
\]
\end{theorem}
\addtocontents{toc}{\SkipTocEntry}
\subsection{An elementary convergence lemma}
We record here the following observation (its proof is a modification of the standard upper semicontinuity property of weak convergence on compact sets \cite[Example~1.63]{AFP}): 

\begin{lemma}\label{lemma:rel_closed} Let $Q\subset \R^n$ be a cube. Let $(\mu_j)_j$ be a sequence of non-negative measures on $Q$ such that $\mu_j \weakstarto \mu$ and $\mu_j(Q) \to \mu(Q)$ as $j \to \infty$ for some finite measure $\mu$ on $Q$.
Then for every relatively closed set $F \subset Q$ it holds 
\[
\mu(F) \ge \limsup_{j \to +\infty} \mu_j(F).  
\]
In particular, for every Borel set $B \subset Q$ such that $\mu(\partial B \cap Q)=0$ it holds 
\[
\lim_{j\to+\infty} \mu_j(B)= \mu(B).
\]
\end{lemma}

\vspace{0.5cm}

\section{An example of non-existence of \texorpdfstring{$\Ksf_0$}{K0} for a \texorpdfstring{$BV$}{BV} function}\label{app:Cantor}

\begin{lemma}
    There exists a purely cantorian $u \in BV(0,1)$, with $|Du|((0,1))=1$, satisfying
    \[
        \liminf_{\e\to 0} \Ksf_\eps(u) = \frac 14 \quad \text{and} \quad \limsup_{\e \to 0} \Ksf_\eps(u) = \frac 12.
    \]
\end{lemma}

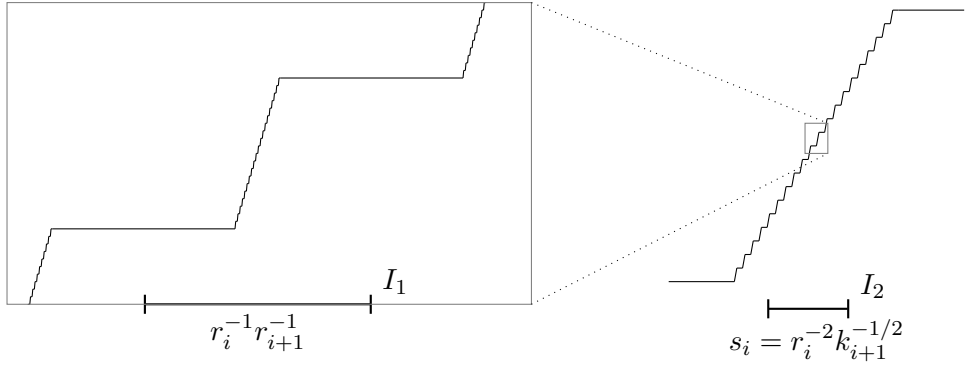
\begin{figure}
\centering
\begin{tikzpicture}[xscale=0.6,yscale=2]

    \def\m{20}
    \def\n{1}
    \def\r{0.3}
    \def\s{5}

    \begin{scope}
    \clip(-5,-0.5) rectangle (6.5,1.5);
    \foreach \j in {-1,...,\n}
    {
    \begin{scope}[shift={({\s*\j},{\j})}]
    \foreach \i in {1,...,\m}
    {
        \draw ({(\i-1)/\m},{(\i-1)/\m})--({(\i-1+\r)/\m},{(\i)/\m});
        \draw ({(\i-1+\r)/\m},{(\i)/\m})--({(\i)/\m},{(\i)/\m});
    };
    \draw({-\s+1},0)--(0,0);
    \end{scope}
    };
    \end{scope}
    
    \draw[thick,|-|] ({(-\s+1)/2},-0.5)--({(\s+1)/2},-0.5) node[above right]{$I_1$};
    \node[below] at (0.5,-0.5) {$r_i^{-1}r_{i+1}^{-1}$};


    \def\p{20}
    \def\n{1}
    \def\r{0.3}
    \def\s{12}


    \begin{scope}[shift={(10.95,-0.35)},scale=1.8,xscale=2]
    \foreach \i in {1,...,\p}
    {
        \draw ({(\i-1)/\p},{(\i-1)/\p})--({(\i-1+\r)/\p},{(\i)/\p});
        \draw ({(\i-1+\r)/\p},{(\i)/\p})--({(\i)/\p},{(\i)/\p});
    };
    \draw({-0.4},0)--(0,0);
    \draw (1,1)--(1.4,1);
    \draw[thick,|-|] (0.2,-0.1)--(0.7,-0.1) node[above right]{$I_2$};
    \node[below] at (0.5,-0.1) {$s_i=r_i^{-2}k_{i+1}^{-1/2}$};
    \end{scope}

    \draw[gray] (-5,-0.5) rectangle (6.5,1.5);
    \draw[gray] (12.5,0.5) rectangle (13,0.7);
    \draw[dotted] (6.5,1.5)--(13,0.7);
    \draw[dotted] (6.5,-0.5)--(13,0.5);
\end{tikzpicture}
\caption{On the left: on the interval $I_1$ of scale $r_i^{-1}k_i^{-1}$ the function is close to a jump. On the right: on the interval $I_2$ of intermediate scale $r_i^{-1}$ the function looks affine.}\label{fig:Cantor}
\end{figure}

\begin{proof}
    The construction consists of several steps.

    \emph{Step 1.} 
    Let $I = [a,b] \subset [0,1]$ be a given interval and let $k$ be a given positive integer. We define $I^k \subset [0,1]$ as the (disjoint) union of  the $k$  closed intervals $\{[a_j, a_j+ k^{-2}]\}_{j = 1}^k$ where the $\{a_j\}_{j = 1}^k$ are equi-distributed (i.e.,  $a_j - a_{j-1} = |I|k^{-1}$) and $a_1 = a$ is the left extreme of $I$.
    Notice that 
    \[
        |I^k| = k \times |I|k^{-2} = \frac{|I|}{k}.
    \]

    \emph{Step 2.} Let $\{k_i\}_{i = 1}^\infty$ be a sequence of positive integers that are growing sufficiently fast, for instance such that 
    \[
        \frac{k_i^2}{k_{i+1}}\le \frac 1{2^i}\,.
    \]
    We define $J_0 \coloneqq [0,1]$ and we define the $J_i$ inductively. Assuming that $J_i$ is a union of $r_i = k_0 \times \cdots \times k_{i}$ disjoint intervals (with $k_0=1$) of the same length, we define $J_{i+1} \subset J_i$ as follows: Denoting by $J_i(r)$, with $r = 1,\dots,r_i$, the disjoint sub-intervals making up $J_i$, we set
    \[
        J_{i+1} = \bigcup_{r = 1}^{r_i} J_i(r)^{k_{i+1}}.
    \]
    The new set $J_{i + 1}$ consists of $r_i \times k_{i+1} = r_{i+1}$ disjoint intervals. Since all $J_i(r)$ have the same length it follows from the definition in Step 1 that all $J_{i+1}(r)$ are also of the same length. Moreover, since $|J_i(r)|^{k_{i+1}} = k_{i+1}^{-1}|J_i(r)|$, we have
    \[
        |J_{i+1}| =  \frac{|J_i|}{k_{i+1}} = \frac{|J_{i-1}|}{k_{i+1} k_{i}} \cdots = \frac{|J_1|}{k_{i+1}\times \cdots \times k_2} = \frac{1}{r_{i+1}}\,.
    \]
    
    \emph{Step 3.} For each $i \in \mathbb N$, we define a probability density $h_i : [0,1] \to \R$ by setting $h_i = r_i \1_{J_i}$. 
    We  then consider its primitive $u_i : [0,1] \to \R$ defined by 
    \[
        u_i(t) \coloneqq \int_0^t h_i(x) \, dx\,,
    \]
    which defines a continuous, non-decreasing function. By construction $u_i(0) = 0$, $u_i(1) = 1$ and $u_i$ has bounded variation $|Du_i|(0,1) = \|h_i\|_{L^1} = 1$. It is not hard to see that $u_{i + 1} 
    \ge u_i$ and
    \[
        u_{i+1}(t) - u_i(t) \le \frac{1}{k_i k_{i+1}}\le \frac1{k_{i + 1}}\,.
    \]
    This, in turn, gives (assuming that $m \ge n$)
    \[
        \|u_m - u_n\|_\infty \le \sum_{i = n+1}^{m} \frac1{k_{i}} \le \sum_{i = n+1}^{m} \frac1{2^{i-1}}\le \frac{1}{2^{n-1}}.
    \]
    It follows that the sequence $\{u_i\}_{i = 1}^\infty$ is Cauchy in $C^0$, and thus, it converges to some continuous $u$ with $u(0)=0$ and $u(1)=1$, which is also non-decreasing. Hence, $u$ is $BV$ and $D^j u=0$.
    Moreover,
    \[
        \spt(|Du|) = J_\infty \coloneqq \bigcap_{i = 1}^\infty J_i\,.
    \]
    Since $|J_\infty| \le |J_i| \to 0$, we further deduce that $D^au = 0$. This proves that $u$ is purely cantorian.

    \emph{Step 4.} 
    By construction, at scale $r_{i}^{-1}r_{i+1}^{-1}$, $u$ is closer and closer to a function that jumps (see interval $I_1$ in Figure \ref{fig:Cantor}), and therefore
    \[
        \lim_{i\to\infty}\Ksf_{{r_{i}^{-1}r_{i+1}^{-1}}}(u)= \frac 12.
    \] 
    On the other side, at { an intermediate scale $s_i$ with $r_{i}^{-1}r_{i+1}^{-1}\ll s_i\ll r_{i}^{-2}$, say $s_i=r_i^{-2}k_{i+1}^{-1/2}$}, $u$ looks pretty much like a piecewise linear map oscillating between constants and affine maps of the same slope (see interval $I_2$ in Figure \ref{fig:Cantor}). Observe that the total contribution of the intervals that contain both the almost-affine part and the constant part is negligible in the limit $i\to \infty$. It follows that
    \[
        \lim_{i\to\infty}\Ksf_{{r_i^{-2}k_{i+1}^{-1/2}}}(u) = \frac14.\qedhere
    \]
\end{proof}

\bibliographystyle{plain}
\bibliography{biblio}

\end{document}